\documentclass[a4paper]{amsart}
\usepackage{amssymb,amsmath}
\usepackage[all]{xy}
\usepackage[colorlinks=true, urlcolor=rltblue] {hyperref}
\usepackage{color}
\definecolor{rltblue}{rgb}{0,0,0.75}
\usepackage{array}
\title{The Veblen functions for computability theorists}

\author{Alberto Marcone}
    \address{Dipartimento di Matematica e Informatica,
    Universit\`{a} di Udine,
    33100 Udine,
    Italy}
\email{alberto.marcone@dimi.uniud.it}
\author{Antonio Montalb\'an}
    \address{Department of Mathematics,
    University of Chicago,
    Chicago, IL 60637,
    USA}
\email{antonio@math.uchicago.edu}
\thanks{Marcone's research was partially supported by PRIN of Italy.
Montalb\'an's research was partially supported by NSF grant DMS-0901169.\\
We thank the referees for their careful reading of the first draft of the paper
and their many suggestions for improving the exposition.}

\date{last saved: July 5, 2010 \\
Compiled: \today}
\newtheorem{theorem}{Theorem}[section]
\newtheorem{lemma}[theorem]{Lemma}

\newtheorem{corollary}[theorem]{Corollary}

\theoremstyle{definition}
\newtheorem{definition}[theorem]{Definition}
\newtheorem{remark}[theorem]{Remark}

\numberwithin{equation}{section}

\newcommand{\set}[2]{\{\,{#1}:{#2}\,\}}

\newcommand{\N}{\ensuremath{\mathbb N}}

\newcommand{\Z}{\ensuremath{\mathbb Z}}
\newcommand{\Bai}{\ensuremath{\N^\N}}
\newcommand{\Seq}{{\ensuremath{\N^{<\N}}}}
\newcommand{\X}{{\ensuremath{\mathcal{X}}}}

\newcommand{\leqt}{\leq_T}

\newcommand{\teq}{\equiv_T}
\newcommand{\converges}{\mathord\downarrow}
\newcommand{\diverges}{\mathord\uparrow}
\newcommand{\la}{\langle}
\newcommand{\ra}{\rangle}
\newcommand{\es}{\emptyset}
\newcommand{\upto}{\!\restriction\!}
\newcommand{\conc}{^\smallfrown}
\newcommand{\KB}{\ensuremath{\leq_{\mathrm{KB}}}}
\newcommand{\KBst}{\ensuremath{<_{\mathrm{KB}}}}

\newcommand{\iexp}[2]{\om^{\la #1, #2 \ra}}
\newcommand{\iexpop}[2]{\omop^{\la #1, #2 \ra}}
\newcommand{\eps}{\varepsilon}
\newcommand{\epsop}{\boldsymbol\varepsilon}
\newcommand{\omop}{{\boldsymbol \omega}}
\newcommand{\phiop}{{\boldsymbol \varphi}}
\newcommand{\PI}[2]{\ensuremath{\boldsymbol\Pi^{#1}_{#2}}}

\newcommand{\system}[1]{\mbox{\fontfamily{cmss}\fontshape{n}\fontseries{m}%
    \selectfont#1}}
\newcommand{\RCA}{\system{RCA}\ensuremath{_0}}
\newcommand{\CA}{\system{CA}\ensuremath{_0}}
\newcommand{\WKL}{\system{WKL}\ensuremath{_0}}
\newcommand{\ACA}{\system{ACA}\ensuremath{_0}}
\newcommand{\ACApr}{\system{ACA}\ensuremath{'_0}}
\newcommand{\ACApl}{\system{ACA}\ensuremath{^+_0}}
\newcommand{\ATR}{\system{ATR}\ensuremath{_0}}
\newcommand{\PCA}{\ensuremath{\Pi^1_1}-\CA}
\def\si{\sigma}
\def\om{\omega}
\def\a{\alpha}
\def\b{\beta}
\def\g{\gamma}
\def\d{\delta}
\newcommand{\J}{\mathcal{J}}
\newcommand{\K}{\mathcal{K}}
\newcommand{\JT}{\mathcal{JT}}
\newcommand{\Jom}[1]{J^{\om^{#1}}}
\newcommand{\JJom}[1]{\J^{\om^{#1}}}
\renewcommand{\hom}[1]{h^{\om^{#1}}}
\newcommand{\JTom}[1]{\JT^{\om^{#1}}}
\newcommand{\Kom}[1]{K^{\om^{#1}}}
\newcommand{\KKom}[1]{\K^{\om^{#1}}}

\def\Y{{\mathcal Y}}
\def\Z{{\mathcal Z}}
\def\isom{\cong}
\def\WO{{\textsf{WOP}}}

\def\bfF{{\mathbf F}}

\begin{document}

\begin{abstract}
We study the computability-theoretic complexity and
proof-the\-o\-ret\-ic strength of the following statements: (1) ``If
$\X$ is a well-ordering, then so is $\epsop_\X$'', and (2) ``If $\X$ is
a well-ordering, then so is $\phiop(\a,\X)$'', where $\a$ is a fixed
computable ordinal and $\phiop$ represents the two-placed Veblen
function. For the former statement, we show that $\om$ iterations of
the Turing jump are necessary in the proof and that the statement is
equivalent to \ACApl\ over \RCA. To prove the latter statement we need
to use $\om^\a$ iterations of the Turing jump, and we show that the
statement is equivalent to $\Pi^0_{\om^\a}$-\CA. Our proofs are purely
computability-theoretic. We also give a new proof of a result of
Friedman: the statement ``if $\X$ is a well-ordering, then so is
$\phiop(\X,0)$'' is equivalent to \ATR\ over \RCA.
\end{abstract}

\maketitle

%\tableofcontents

\section{Introduction}\label{sect:intro}

The Veblen functions on ordinals are well-known and commonly used in
proof theory. Proof theorists know that these functions have an
interesting and complex behavior that allows them to build ordinals
that are large enough to calibrate the consistency strength of
different logical systems beyond Peano Arithmetic. The goal of this
paper is to investigate this behavior from a computability viewpoint.

The well-known ordinal $\eps_0$ is defined to be the first fixed point
of the function $\a\mapsto\om^\a$, or equivalently $\eps_0 = \sup\{\om,
\om^\om, \om^{\om^{\om}},\dots\}$. In 1936 Gentzen \cite{Gen36}, used
transfinite induction on primitive recursive predicates along $\eps_0$,
together with finitary methods, to give a proof of the consistency of
Peano Arithmetic. This, combined with G\"{o}del's Second Incompleteness
Theorem, implies that Peano Arithmetic does not prove that $\eps_0$ is
a well-ordering. On the other hand, transfinite induction up to any
smaller ordinal can be proved within Peano Arithmetic. This makes
$\eps_0$ the {\em proof-theoretic ordinal} of Peano Arithmetic.

This result kicked off a whole area of proof theory, called ordinal
analysis, where the complexity of logical systems is measured in terms
of (among other things) how much transfinite induction is needed to
prove their consistency. (We refer the reader to \cite{Rat06} for an
exposition of the general ideas behind ordinal analysis.) The
proof-theoretic ordinal of many logical systems have been calculated.
An example that is relevant to this paper is the system \ACApl\ (see
Section \ref{ssect:subsystems} below), whose proof-theoretic ordinal is
$\varphi_2(0)=\sup\{\eps_0, \eps_{\eps_0},
\eps_{\eps_{\eps_0}},\dots\}$; the first fixed point of the {\em
epsilon function} \cite[Thm.\ 3.5]{Rat91}. The {\em epsilon function}
is the one that given $\g$, returns $\eps_\g$, the $\g$th fixed point
of the function $\a\mapsto\om^\a$ starting with $\g=0$.

The Veblen functions, introduced in 1908 \cite{Veb08}, are functions
on ordinals that are commonly used in proof theory to obtain the
proof-theoretic ordinals of predicative theories beyond Peano
Arithmetic.
\begin{itemize}
\item $\varphi_0(\a)=\om^\a$.
\item $\varphi_{\b+1}(\a)$ is the $\a$th fixed point of $\varphi_\b$ starting with $\a=0$.
\item when $\lambda$ is a limit ordinal, $\varphi_{\lambda}(\a)$ is
    the $\a$th simultaneous fixed point of all the $\varphi_\b$ for
    $\b<\lambda$, also starting with $\a=0$.
\end{itemize}
Note that $\varphi_1$ is the epsilon function.

The Feferman-Sch\"{u}tte ordinal $\Gamma_0$ is defined to be the least
ordinal closed under the binary Veblen function
$\varphi(\b,\a)=\varphi_\b(\a)$, or equivalently
\[
\Gamma_0=\sup\{\varphi_0(0), \varphi_{\varphi_0(0)}(0),
\varphi_{\varphi_{\varphi_0(0)}(0)}(0),\dots\}.
\]
$\Gamma_0$ is the proof-theoretic ordinal of Feferman's Predicative
Analysis \cite{Fef64,Sch77}, and of \ATR\ \cite{FMcS}\footnote{for
the definition of \ATR\ and of other subsystems of second order
arithmetic mentioned in this introduction see Section
\ref{ssect:subsystems} below.}. Again, this means that the
consistency of \ATR\ can be proved by finitary methods together with
transfinite induction up to $\Gamma_0$, and that \ATR\ proves the
well-foundedness of any ordinal below $\Gamma_0$.

Sentences stating that a certain linear ordering is well-ordered are
\PI11. So, even if they are strong enough to prove the consistency of
some theory, they have no set-existence implications. However, a
sentence stating that an operator on linear orderings preserves
well-orderedness is \PI12, and hence gives rise to a natural reverse
mathematics question. The following theorems answer two questions of
this kind.

\begin{theorem}[Girard, {\cite[p.\ 299]{Girard}}]\label{Girard}
Over \RCA, the statement \lq\lq if \X\ is a well-ordering then
$\omop^\X$ is also a well-ordering\rq\rq\ is equivalent to \ACA.
\end{theorem}

\begin{theorem}[H. Friedman, unpublished]\label{Friedman}
Over \RCA, the statement \lq\lq if \X\ is a well-ordering then
$\phiop(\X,0)$ is a well-ordering\rq\rq\ is equivalent to \ATR.
\end{theorem}

Let $\bfF$ be an operator on linear orderings. We consider the statement
\[
\WO(\bfF): \quad \forall \X\ (\X\text{ is a well-ordering} \implies \bfF(\X)\text{ is a well-ordering}).
\]

We study the behavior of $\bfF$ by analyzing the computational
complexity of the proof of $\WO(\bfF)$ as follows. The statement
$\WO(\bfF)$ can be restated as ``if $\bfF(\X)$ has a descending
sequence, then $\X$ has a descending sequence to begin with''. Given
$\bfF$, the question we ask is:
\begin{quote}
Given a linear ordering $\X$ and a descending sequence in $\bfF(\X)$,
how difficult is to build a descending sequence in $\X$?
\end{quote}

From Hirst's proof of Girard's result \cite{Hirst94}, we can extract
the following answer for $\bfF(\X)=\omop^\X$.

\begin{theorem}\label{Hirst}
If $\X$ is a computable linear ordering, and $\omop^{\X}$ has a
computable descending sequence, then $0'$ computes a descending
sequence in $\X$. Furthermore, there exists a computable linear
ordering $\X$ with a computable descending sequence in $\omop^{\X}$
such that every descending sequence in $\X$ computes $0'$.
\end{theorem}

The first statement of the theorem follows from the results of Section
\ref{sect:forward}, which includes the upper bounds of the
computability-theoretic results and the \lq\lq forward
directions\rq\rq\ of the reverse mathematics results. We include a
proof of the second statement in Section \ref{sect:exp}, where we
modify Hirst's idea to be able to apply it on our other results later.
In doing so, we give a new definition of the Turing jump which,
although computationally equivalent to the usual jump, is
combinatorially easier to manage. This allows us to define computable
approximations to the Turing jump, and we can also define a computable
operation on trees that produces trees whose paths are the Turing jumps
of the input tree. Furthermore, our definition of the Turing jump
behaves nicely when we take iterations.

In Section \ref{sect:eps} we use these features of our proof of
Theorem \ref{Hirst}. First, in Section \ref{ssect:finite} we consider
finite iterations of the Turing jump and of ordinal exponentiation.
(We write $\iexpop n\X$ for the $n$th iterate of the operation
$\omop^\X$; see Definition \ref{iexpop}.) In Theorem \ref{0^n}, we
prove:

\begin{theorem}
Fix $n\in\N$. If $\X$ is a computable linear ordering, and $\iexpop
n\X$ has a computable descending sequence, then $0^{(n)}$ computes a
descending sequence in $\X$. Conversely, there exists a computable
linear ordering $\X$ with a computable descending sequence in $\iexpop
n\X$ such that the jump of every descending sequence in $\X$ computes
$0^{(n)}$.
\end{theorem}

From this, in Section \ref{ssect:rmepsilon}, we obtain the following
reverse mathematics result.

\begin{theorem}\label{thm:ACApr}
Over \RCA, $\forall n\, \WO(\X\mapsto\iexpop n\X)$ is equivalent to
\ACApr.
\end{theorem}

The first main new result of this paper is obtained in Section
\ref{ssect:eps} and analyzes the complexity behind the epsilon
function.

\begin{theorem}
If $\X$ is a computable linear ordering, and $\epsop_\X$ has a
computable descending sequence, then $0^{(\om)}$ can compute a
descending sequence in $\X$. Conversely, there is a computable linear
ordering $\X$ with a computable descending sequence in $\epsop_\X$ such
that the jump of every descending sequence in $\X$ computes
$0^{(\om)}$.
\end{theorem}

We prove this result in Theorems \ref{thm:epsilon forward} and
\ref{omega}. Then, as a corollary of the proof, we obtain the following
result in Section \ref{ssect:rmepsilon}.

\begin{theorem}\label{thm: espilon vs ACApl}
Over \RCA, $\WO(\X\mapsto\epsop_\X)$ is equivalent to \ACApl.
\end{theorem}

Our proof is purely computability-theoretic and plays with the
combinatorics of the $\om$-jump and the epsilon function. By
generalizing the previous ideas, we obtain a new definition of the
$\om$-Turing jump, which we can also approximate by a computable
function on finite strings and by a computable operator on trees. An
important property of our $\om$-Turing jump operator is that it is
essentially a fixed point of the jump operator: for every real $Z$, the
$\om$-Turing jump of $Z$ is equal to the $\om$-Turing jump of the jump
of $Z$, except for the first bit (we mean equal as sequences of
numbers, not only Turing equivalent). Notice the analogy with the
$\epsop$ and $\omop$ operators.

After a draft of the proof of Theorem \ref{thm: espilon vs ACApl} was
circulated, Afshari and Rathjen \cite{AR} gave a completely different
proof using only proof-theoretic methods like cut-elimination, coded
$\om$-models and Sch\"{u}tte deduction chains. They prove that
$\WO(\X\mapsto\epsop_\X)$ implies the existence of countable coded
$\om$-models of \ACA\ containing any given set, and that this in turn
is equivalent to \ACApl. To this end they prove a completeness-type
result: given a set $Z$, they can either build an $\om$-model of \ACA\
containing $Z$ as wanted, or obtain a proof tree of `0=1' in a suitable
logical system with formulas of rank at most $\om$. The latter case
leads to a contradiction as follows. The logical system where we get
the proof tree has cut elimination, increasing the rank of the proof
tree by an application of the $\epsop$ operator. Using
$\WO(\X\mapsto\epsop_\X)$, $\X$ being the Kleene-Brouwer ordering on
the proof tree of `0=1', they obtain a well-founded cut-free proof tree
of `0=1'.

In Section \ref {sect:General case}, we move towards studying the
computable complexity of the Veblen functions. Given a computable
ordinal $\a$, we calibrate the complexity of
$\WO(\X\mapsto\phiop(\a,\X))$ with the following result, obtained by
extending our definitions to $\om^\a$-Turing jumps.

\begin{theorem}\label{thm: veblen vs jumps}
Let $\a$ be a computable ordinal. If $\X$ is a computable linear
ordering, and $\phiop(\a,{\X})$ has a computable descending sequence,
then $0^{(\om^\a)}$ computes a descending sequence in $\X$. Conversely,
there is a computable linear ordering $\X$ such that $\phiop(\a,{\X})$
has a computable descending sequence but every descending sequence in
$\X$ computes $0^{(\om^\a)}$.
\end{theorem}

This result will follow from Theorem \ref{thm:veblen forward} and
Theorem \ref{alpha}. In Section \ref{ssect:rmVeblen}, as a corollary,
we get the following result.

\begin{theorem}\label{thm: veblen vs PiAlphaCA}
Let $\a$ be a computable ordinal. Over \RCA,
$\WO(\X\mapsto\phiop(\a,\X))$ is equivalent to $\Pi^0_{\om^\a}$-\CA.
\end{theorem}

Exploiting the uniformity in the proof of Theorem \ref{thm: veblen vs
jumps}, we also obtain a new purely computability-theoretic proof of
Friedman's result (Theorem \ref{Friedman}). Before our proof, Rathjen
and Weiermann \cite{RW} found a new, fully proof-theoretic proof of
Friedman's result. They use a technique similar to the proof of
Afshari and Rathjen mentioned above. Friedman's original proof has
two parts, one computability-theoretic and one proof-theoretic.

The table below shows the systems studied in this paper (with the
exception of \ACApr). The second column gives the proof-theoretic
ordinal of the system, which were calculated by Gentzen, Rathjen,
Feferman, and Sch\"{u}tte. The third column gives the operator $\bfF$
on linear orderings such that $\WO(\bfF)$ is equivalent to the given
system. The last column gives references for the different proofs of
these equivalences in historical order ([MM] refers to this paper).

\begin{center}
\begin{tabular}{|c|c|c|>{\Small}l|}
\hline
System & p.t.o. & $\bfF(\X)$ & references \\
\hline
\ACA & $\eps_0$ & $\omop^\X$ & Girard \cite{Girard}; Hirst \cite{Hirst94}\\
\ACApl & $\varphi_2(0)$ & $\epsop_\X$ & [MM]; Afshari-Rathjen \cite{AR}\\
$\Pi^0_{\om^\a}$-\CA &  $\varphi_{\a+1}(0)$ & $\phiop(\a,\X)$ & [MM]\\
\ATR  & $\Gamma_0$ & $\phiop(\X,0)$ &Friedman \cite{FMW}; Rathjen-Weiermann \cite{RW}; [MM]\\
\hline
\end{tabular}
\end{center}

Notice that in every case, the proof-theoretic ordinal equals
\[
\sup\{\bfF(0), \bfF(\bfF(0)), \bfF(\bfF(\bfF(0))),\dots\}.
\]

%%%%%%%%%%%%%%%%%%%%%%%%%%%
\section{Background and definitions}\label{sect:background}

%%%%%%%%%%%%%%%%%%%%%%%%%%%
\subsection{Veblen operators and ordinal notation}\label{ssect:Veblen operators}

We already know what the $\om$, $\eps$ and $\varphi$ functions do on
ordinals. In this section we define operators $\omop$, $\epsop$ and
$\phiop$, that work on all linear orderings. These operators are
computable, and when they are applied to a well-ordering, they coincide
with the $\om$, $\eps$ and $\varphi$ functions on ordinals.

To motivate the definition of $\omop^\X$ we use the following
observation due to Cantor \cite{Can97}. Every ordinal below $\om^\a$
can be written in a unique way as a sum
\[
\om^{\b_0}+\om^{\b_1}+\dots+\om^{\b_{k-1}},
\]
where $\a>\b_0\geq \b_1\geq \dots\geq \b_{k-1}$.

\begin{definition}
Given a linear ordering $\X$, $\omop^\X$ is defined as the set of
finite strings  $\la x_0, x_1,\dots, x_{k-1}\ra \in \X^{<\om}$
(including the empty string) where $x_0\geq_\X x_1\geq_\X \dots \geq_\X
x_{k-1}$. We think of $\la x_0, x_1,\dots, x_{k-1}\ra \in \omop^\X $ as
$\om^{x_0}+\om^{x_1}+\dots+\om^{x_{k-1}}$. The ordering on $\omop^\X$
is the lexicographic one: $\la x_0, x_1,\dots,
x_{k-1}\ra\leq_{\omop^\X}\la y_0, y_1,\dots, y_{l-1}\ra$ if either
$k\leq l$ and $x_i=y_i$ for every $i<k$, or for the least $i$ such that
$x_i\neq y_i$ we have $x_i<_\X y_i$.
\end{definition}

We use the following notation for the iteration of the $\omop$
operator.

\begin{definition}\label{iexpop}
Given a linear ordering \X, let $\iexpop 0\X = \X$ and $\iexpop {n+1}\X
= \omop^{\iexpop n\X}$.
\end{definition}

To motivate the definition of the $\epsop$ operator we start with the
following observations. On the ordinals, the closure of the set $\{0\}$
under the operations $+$ and $t\mapsto \om^t$, is the set of the
ordinals strictly below $\eps_0$. The closure of $\{0,\eps_0\}$ under
the same operations, is the set of the ordinals strictly below
$\eps_1$. In general, if we take the closure of $\{0\} \cup
\set{\eps_\b}{\b<\a}$ we obtain all ordinals strictly below $\eps_\a$.

\begin{definition}
Let $\X$ be a linear ordering. We define $\epsop_\X$ to be the set of
formal terms defined as follows:
\begin{itemize}
\item $0$ and $\eps_x$, for $x\in\X$, belong to $\epsop_\X$, and
    are called \lq\lq constants\rq\rq,
\item if $t_1,t_2\in \epsop_\X$, then $t_1+t_2\in\epsop_\X$,
\item if $t\in \epsop_\X$, then $\om^t\in \epsop_\X$.
\end{itemize}
Many of the terms we defined represent the same element, so we need to
find normal forms for the elements of $\epsop_\X$. The definition of
the ordering on $\epsop_\X$ is what one should expect when $\X$ is an
ordinal. We define the normal form of a term and the relation
$\leq_{\epsop_\X}$ simultaneously by induction on terms.

We say that a term $t=t_0+\dots+t_k$ is in  {\em normal form} if either
$t=0$ (i.e.\ $k=0$ and $t_0=0$), or the following holds: (a)
$t_0\geq_{\epsop_\X} t_1\geq_{\epsop_\X}\dots\geq_{\epsop_\X} t_k> 0$,
and (b) each $t_i$ is either a constant or of the form $\om^{s_i}$,
where $s_i$ is in normal form and $s_i\neq\eps_x$ for any $x$.

Every $t\in\epsop_\X$ can be written in normal form by applying the
following rules:
\begin{itemize}
\item $+$ is associative,
\item $s+0=0+s=s$,
\item if $s<_{\epsop_\X} r$, then $\om^s+\om^r=\om^r$,
\item $\om^{\eps_x}=\eps_x$.
\end{itemize}

Given $t=t_0+\dots+t_k$ and $s=s_0+\dots+s_l$ in normal form, we let
$t\leq_{\epsop_\X} s$ if one of the following conditions apply
\begin{itemize}
\item $t=0$,
\item $t=\eps_x$ and, for some $y\geq_\X x$, $\eps_y$ occurs in
    $s$,
\item $t=\om^{t'}$, $s_0=\eps_y$ and $t'\leq_{\epsop_\X}\eps_y$,
\item $t=\om^{t'}$, $s_0=\om^{s'}$ and $t'\leq_{\epsop_\X}s'$,
\item $k>0$ and $t_0<_{\epsop_\X}s_0$,
\item $k>0$, $t_0=s_0$, $l>0$ and $t_1+\dots+t_k \leq_{\epsop_\X}
    s_1+\dots+s_l$.
\end{itemize}
\end{definition}

The observation we made before the definition shows how the $\epsop$
operator coincides with the $\eps$-function when $\X$ is an ordinal
(this includes the case $\X=\es$, when $0$ is the only constant and we
obtain $\eps_0$ as expected).

\begin{definition}\label{iexp}
In analogy with Definition \ref{iexpop}, for $t \in \epsop_\X$ we use
$\iexp nt$ to denote the term in $\epsop_\X$ obtained by applying the
$\om$ function symbol $n$ times to $t$.
\end{definition}

\begin{definition}
If $\X$ is a linear ordering and $x \in \X$, let $\X\upto x$ be the
linear ordering with domain $\set{y\in\X}{y<_\X x}$.
\end{definition}

The following lemma expresses the compatibility of the $\omop$ and
$\epsop$ operators.

\begin{lemma}\label{lemma:om eps}
If $\X$ is a linear ordering, then for every $t\in \epsop_\X$ and $n
\in \N$
\[
\iexpop n{\epsop_\X\upto t} \isom \epsop_\X \upto \iexp nt
\]
via a computable isomorphism. In particular, $\omop^{\epsop_\X\upto t}
\isom \epsop_\X \upto \om^t$.
\end{lemma}
\begin{proof}
The proof is by induction on $n$. When $n=0$ the identity is the
required isomorphism. If $\psi: \iexpop n{\epsop_\X\upto t} \to
\epsop_\X \upto \iexp nt$ is an isomorphism, then the function mapping
the empty string to $0$ and $\la t_0,\dots,t_k\ra$ to
$\om^{\psi(t_0)}+\dots+\om^{\psi(t_k)}$ witnesses $\iexpop
{n+1}{\epsop_\X\upto t} \isom \epsop_\X \upto \iexp {n+1}t$.
\end{proof}

To define the $\phiop$ operator we start with the following
observations. If we take the closure of the set $\{0\}$ under the
operations $+$, $t\mapsto \om^t$ and $t\mapsto \eps_t$, we get all the
ordinals up to $\varphi_2(0)$. If we take the closure of
$\{0\}\cup\set{\varphi_2(\b)}{\b<\a}$ we get all the ordinals below
$\varphi_2(\a)$. In general, we obtain $\varphi_\g(\a)$ as the closure
of $\{0\} \cup \set{\varphi_\g(\b)}{\b<\a}$ under the operations $+$,
and $t\mapsto \varphi_\delta(t)$, for all $\delta<\g$.

\begin{definition}\label{def:phiop}
Let $\X$ and $\Y$ be linear orderings. We define $\phiop(\Y,\X)$ to
be the set of formal terms defined as follows:
\begin{itemize}
\item $0$ and $\varphi_{\Y,x}$, for $x\in\X$, belong to
    $\phiop(\Y,\X)$, and are called \lq\lq constants\rq\rq,
\item if $t_1,t_2\in \phiop(\Y,\X)$, then $t_1+t_2\in\phiop(\Y,\X)$,
\item if $t\in \phiop(\Y,\X)$ and $\d\in\Y$, then $\varphi_\d(t)
    \in \phiop(\Y,\X)$.
\end{itemize}
We define the normal form of a term and the relation $\leq_{\phiop(\Y,\X)}$
simultaneously by induction on terms. We write $\leq_\varphi$ instead
of $\leq_{\phiop(\Y,\X)}$ to simplify the notation.

We say that a term $t=t_0+\dots+t_k$ is in  {\em normal form} if either
$t=0$, or the following holds: (a) $t_0\geq_\varphi
t_1\geq_\varphi\dots\geq_\varphi t_k> 0$, and (b) each $t_i$ is either
a constant or of the form $\varphi_\d(s_i)$, where $s_i$ is in normal
form and $s_i\neq\varphi_{\d'}(s_i')$ for $\d'>\d$.

Every $t\in\phiop(\Y,\X)$ can be written in normal form by applying the
following rules:
\begin{itemize}
\item $+$ is associative,
\item $s+0=0+s=s$,
\item if $\varphi_{\d'}(s)<_\varphi \varphi_\d(r)$, then $\varphi_{\d'}(s)+\varphi_\d(r)=\varphi_\d(r)$.
\item if $\d ' >\d$, then $\varphi_\d( \varphi_{\d'}(r))=\varphi_{\d'}(r)$.
\item if $\d\in\Y$ then $\varphi_\d(\varphi_{\Y,r})=
    \varphi_{\Y,r}$.
\end{itemize}
The motivation for the last two items is that if $\d ' >\d$, anything
in the image of $\varphi_{\d'}$ is a fixed point of $\varphi_\d$.

Given $t=t_0+\dots+t_k$ and $s=s_0+\dots+s_l$ in normal form, we let
$t\leq_\varphi s$ if one of the following conditions apply
\begin{itemize}
\item $t=0$,
\item $t=\phiop_{\Y,x}$ and, for some $y\geq_\X x$,
    $\varphi_{\Y,y}$ occurs in $s$,
\item $t=\varphi_\d(t')$, $s_0=\varphi_{\d'}(s')$ and
$\begin{cases}
\d<\d'   \text{ and }   t' \leq_\varphi \varphi_{\d'}(s'),  \text{ or }  \\
\d=\d'   \text{ and }   t' \leq_\varphi s',  \text{ or }    \\
\d>\d'   \text{ and }   \varphi_\d(t') \leq_\varphi s',
\end{cases}$
\item $k>0$ and $t_0<_\varphi s_0$,
\item $k>0$, $t_0=s_0$, $l>0$ and $t_1+\dots+t_k \leq_\varphi
    s_1+\dots+s_l$.
\end{itemize}
\end{definition}

%%%%%%%%%%%%%%%%%%%%
\subsection{Notation for strings and trees}\label{ssect:notation}

Here we fix our notation for sequences (or strings) of natural numbers.
The \emph{Baire space \Bai} is the set of all infinite sequences of
natural numbers. As usual, an element of \Bai\ is also called a
\emph{real}. If $X \in \Bai$ and $n \in \N$, $X(n)$ is the $(n+1)$-st
element of $X$. \Seq\ is the set of all finite strings of natural
numbers. When $\si \in \Seq$ we use $|\si|$ to denote its \emph{length}
and, for $i<|\si|$, $\si(i)$ to denote its $(i+1)$-st element. We write
$\es$ for the \emph{empty string} (i.e.\ the only string of length
$0$), and $\la n \ra$ for the string of length $1$ whose only element
is $n$. When $\si, \tau \in \Seq$, $\si \subseteq \tau$ means that
$\si$ is an \emph{initial segment} of $\tau$, i.e.\ $|\si| \leq |\tau|$
and $\si(i) = \tau(i)$ for each $i<|\si|$. We use $\si \subset \tau$ to
mean $\si \subseteq \tau$ and $\si \neq \tau$. If $X \in \Bai$ we write
$\si \subset X$ if $\si(i) = X(i)$ for each $i<|\si|$. We use $\si
\conc \tau$ to denote the \emph{concatenation} of $\si$ and $\tau$,
that is the string $\rho$ such that $|\rho| = |\si|+|\tau|$, $\rho(i)=
\sigma(i)$ when $i<|\si|$, and $\rho(|\si|+i)=\tau(i)$ when $i<|\tau|$.
If $X \in \Bai$, $\si \in \Seq$ and $t \in \N$, $X \upto t$ is the
initial segment of $X$ of length $t$, while $\si \upto t$ is the
initial segment of $\si$ of length $t$ if $t\leq|\si|$, and $\si$
otherwise.

We fix an enumeration of \Seq, so that each finite string is also a
natural number, and hence can be an element of another string. This
enumeration is such that all the operations and relations discussed in
the previous paragraph are computable. Moreover we can assume that $\si
\subset \tau$ (as strings) implies $\si<\tau$ (as natural numbers). For
an enumeration with these properties see e.g.\ \cite[\S II.2]{Sim99}.

The following operation on strings will be useful.

\begin{definition}\label{def:ell}
If $\si\in\Seq$ is nonempty let $\ell(\si) = \la \si(|\si|-1) \ra$, the
string of length one whose only entry is the last entry of $\si$.
\end{definition}

\begin{definition}
A \emph{tree} is a set $T \subseteq \Seq$ such that $\si \upto t \in T$
whenever $\si \in T$ and $t<|\si|$. If $T$ is a tree, $X \in \Bai$ is
\emph{a path through $T$} if $X \upto t \in T$ for all $t$. We let
$[T]$ be the set of all paths through $T$.
\end{definition}

\begin{definition}\label{def:Tsigma}
If $T$ is a tree and $\si \in \Seq$ we let $T_\si = \set{\rho \in
T}{\rho \subseteq \si \lor \si \subseteq \rho}$.
\end{definition}

\begin{definition}
\KB\ is the usual \emph{Kleene-Brouwer ordering} of $\Seq$: if $\si,
\tau \in \Seq$, we let $\si \KB \tau$ if either $\si \supseteq \tau$ or
there is some $i$ such that $\si \upto i = \tau \upto i$ and $\si(i) <
\tau (i)$.
\end{definition}

The following is well-known (see e.g.\ \cite[Lemma V.1.3]{Sim99}).

\begin{lemma}\label{lemma:KB}
Let $T \subseteq \Seq$ be a tree: $T$ is well-founded (i.e.\ $[T]=\es$)
if and only if the linear ordering $(T, {\KB})$ is well-ordered.
Moreover, if $f\colon \N \to T$ is a descending sequence with respect
to \KB, there exists $Y \in [T]$ such that $Y \leqt f'$.
\end{lemma}

We will need some terminology to describe functions between partial
orderings.

\begin{definition}
Let $f\colon P \to Q$ be a function, $\leq_P$ and $\leq_Q$ be partial
orderings of $P$ and $Q$ respectively, with $<_P$ and $<_Q$ the
corresponding strict orderings. We say that $f$ is \emph{$({<_P},
{<_Q})$-monotone} if for every $x, y \in P$ such that $x <_P y$ we have
$f(x) <_Q f(y)$.
\end{definition}

%%%%%%%%%%%%%%%%%%%%
\subsection{Computability theory notation}\label{ssect:ctnotation}

We use standard notation from computability theory. In particular, for
a string $\si\in \N^{\leq\N}$, $\{e\}^\si (n)$ denotes the output of
the $e$th Turing machine on input $n$, run with oracle $\si$, for at
most $|\si|$ steps (where $|\si| = \infty$ when $\si \in \Bai$). If
this computation does not halt in less than $|\si|$ steps we write
$\{e\}^\si (n) \diverges$, otherwise we write $\{e\}^\si (n)
\converges$. We write $\{e\}^\si_t (n) \converges$ if the computation
halts in less than $\min(|\si|, t)$ steps.

Given $X,Y\subseteq\N$, the predicate $X=Y'$ is defined as usual:
\[
X=Y'\iff \forall e(e\in X\leftrightarrow \{e\}^Y(e) \converges).
\]

\begin{definition}
Given an ordinal $\b$ (or actually any presentation of a linear
ordering with first element 0), we say that $X=Y^{(\b)}$ if
\[
X^{[0]}=Y \ \ , \ \  \forall\g<\b\ (X^{[\g]}={X^{[<\g]}}')\ \text{ and } X=X^{[<\b]}.
\]
where $X^{[\g]}=\set{y}{\la \g,y\ra\in X}$ and $X^{[<\g]} = \set{\la
\d,y \ra}{\d<\g\ \&\ \la \d,y \ra \in X}$.
\end{definition}

%%%%%%%%%%%%%%%%%%%%
\subsection{Subsystems of second order arithmetic}\label{ssect:subsystems}

We refer the reader to \cite{Sim99} for background information on
subsystems of second order arithmetic. All subsystems we consider
extend \RCA\ which consists of the axioms of ordered semi-ring, plus
$\Delta^0_1$-comprehension and $\Sigma^0_1$-induction. Adding
set-existence axioms to \RCA\ we obtain \WKL, \ACA, \ATR,\ and \PCA,
completing the so-called \lq\lq big five\rq\rq\ of reverse mathematics.

In this paper we are interested in \ACA, \ATR, and some theories which
lie between these two. All these theories can be presented in terms of
\lq\lq jump-existence axioms\rq\rq, as follows:
\begin{description}
\item[\ACA] \RCA + $\forall Y\exists X\ (X=Y')$
\item[\ACApr] \RCA + $\forall Y \forall n \exists X\ (X=Y^{(n)})$
\item[\ACApl] \RCA + $\forall Y\exists X\ (X=Y^{(\om)})$
\item[$\Pi^0_\b$-\CA] \RCA + $\b \text{ well-ordered} \land \forall Y\exists X\ (X=Y^{(\b)})$,\\
    where $\b$ is a presentation of a computable
    ordinal\footnote{The system $\Pi^0_\b$-\CA\ is sometimes
    denoted by $(\Pi^0_1\text{-\CA})_\b$ in the literature.}
\item[\ATR] \RCA + $\forall \a(\a\text{ well-ordered}\implies
    \forall Y\exists X\ (X=Y^{(\a)}))$
\end{description}
Notice that $\Pi^0_1$-\CA\ is \ACA\ and $\Pi^0_\om$-\CA\ is \ACApl.
$\Pi^0_\b$-\CA\ is strictly stronger than $\Pi^0_\g$-\CA\ if and only
if $\b\geq\g\cdot\om$. In fact the $\om$-model $\bigcup_{\a<\g\cdot\om}
\set{X}{X \leqt 0^{(\a)}}$ satisfies $\Pi^0_\a$-\CA\ for all
$\a<\g\cdot\om$, but not $\Pi^0_{\g\cdot\om}$-\CA. Each theory in the
above list is strictly stronger than the preceding ones if we assume
$\b \geq \om^2$.

\ACA\ and \ATR\ are well-known and widely studied: \cite{Sim99}
includes a chapter devoted to each of them and their equivalents.
(The axiomatization of \ATR\ given above is equivalent to the usual
one by \cite[Theorem VIII.3.15]{Sim99}.) \ACApl\ was introduced in
\cite{BHS}, where it was shown that it proves Hindman's Theorem in
combinatorics (to this day it is unknown whether \ACApl\ and
Hindman's Theorem are equivalent). \ACApl\ has also been used in
\cite{Shore06} (where it is proved that \ACApl\ is equivalent to
statements asserting the existence of invariants for Boolean
algebras) and in \cite{MM} (where \ACApl\ is used to prove a
restricted version of Fra\"{\i}ss\'{e}'s conjecture on linear orders). \ACApr\
is also featured in \cite{MM}. The computation of its proof-theoretic
ordinal, which turns out to be $\eps_\om$, is due to J\"{a}ger
(unpublished notes, a proof appears in \cite{McA}, and a different
proof is included in \cite{Af-thesis}). The theories $\Pi^0_\b$-\CA\
are natural generalizations of \ACApl.

%%%%%%%%%%%%%%%%%%%%%%%%%%%%%%%%%%%%%%%%%%%%%%%%%%%%%%%%%%%%%%%%%%%%%%%%%%%%%%%%%%%%%
\section{Forward direction}\label{sect:forward}

In this section we prove the \lq\lq forward direction\rq\rq of Theorems
\ref{Girard}, \ref{thm:ACApr}, \ref{thm: espilon vs ACApl}, \ref{thm:
veblen vs PiAlphaCA}, and \ref{Friedman}. The results in this section
are already known (though often written in different settings) but we
include them as our proofs illustrate how the iterates of the Turing
jump relate with the epsilon and Veblen functions.

The following theorem is essentially contained in Hirst's proof
\cite{Hirst94} of the closure of well-orderings under exponentiation
in \ACA.

\begin{theorem}\label{thm:om forward}
If $\X$ is a $Z$-computable linear ordering, and $\omop^{\X}$ has a
$Z$-computable descending sequence, then $Z'$ can compute a descending
sequence in $\X$.
\end{theorem}
%\begin{proof}
%Let $(a_k:k\in\N)$ be a $Z$-computable descending sequence in
%$\omop^\X$. Write $a_k$ as $\om^{x_{k,0}}+ \om^{x_{k,1}}+ \dots+
%\om^{x_{k,l_k}}$. The sequence $(x_{k,0}:k\in\N)$ is non-increasing.
%So, either it has a strictly descending subsequence, or it is
%eventually constant. In the former case, we can $Z$-computably find a
%descending subsequence and we are done. In the latter case, $Z'$ can
%find $n_0$ such that $\forall k\geq n_0\ (x_{k,0}=x_{n_0,0})$. Then,
%the sequence $(x_{k,1}:k\in\N, k\geq n_0)$ is non-increasing. Again,
%either we find a $Z$-computable descending subsequence, or it is
%constant after some $x_{n_1,1}$. If we continue with this process,
%either at some point we find a $Z$-computable descending subsequence of
%some $(x_{k,i+1}:k\in\N, k\geq n_{i})$, or we never do and we get a
%$Z'$-computable non-increasing sequence $(x_{n_i,i}:i\in\N)$. If this
%sequence is constant after some $x_{n_j,j}$ then the $a_{n_i}$ with
%$i>j$ contain more and more terms all equal to $x_{n_j,j}$. Since
%$a_{n_j}$ is greater than each $a_{n_i}$ with $i>j$, this requires the
%length of $a_{n_j}$ to be infinite, which is impossible. Thus
%$(x_{n_i,i}:i\in\N)$ contains a $Z'$-computable descending subsequence.
%\end{proof}
\begin{proof}
Let $(a_k:k\in\N)$ be a $Z$-computable descending sequence in
$\omop^\X$. We can write $a_k$ in the form $\om^{x_{k,0}} \cdot
m_{k,0} + \om^{x_{k,1}} \cdot m_{k,1}+ \dots+ \om^{x_{k,l_k}} \cdot
m_{k,l_k}$ where each $m_{k,0}\in\N$ is positive and $x_{k,i} >_\X
x_{k,i+1}$ for all $i<l_k$.

Using $Z'$, we recursively define a function $f: \N \to \X \times \om$
which is decreasing with respect to the lexicographic ordering $<_{\X
\times \om}$. (We use $x\cdot m$ to denote $\la x,m\ra\in\X \times
\om$.) Each $f(n)$ is of the form $x_{k,i} \cdot m_{k,i}$ for some $k$
and $i\leq l_k$. At the following step, when we define  $f(n+1)$,
either we increase $k$ and leave $i$ unchanged, or, if this is not
possible, we keep $k$ unchanged and increase $i$ by one. We will have
that if $f(n)$ is of the form $x_{k,i} \cdot m_{k,i}$, then $x_{h,j}
\cdot m_{h,j} = x_{k,j} \cdot m_{k,j}$ for all $h>k$ and $j<i$.

Let $f(0) = x_{0,0} \cdot m_{0,0}$. Assuming we already defined $f(n) =
x_{k,i} \cdot m_{k,i}$, we need to define $f(n+1)$. If there exist
$h>k$ such that $x_{h,i} \cdot m_{h,i} <_{\X \times \om} x_{k,i} \cdot
m_{k,i}$, then let $f(n+1)=x_{h,i} \cdot m_{h,i}$ for the least such
$h$. If $x_{h,i} \cdot m_{h,i} \geq_{\X \times \om} x_{k,i} \cdot
m_{k,i}$ for all $h>k$ then we must have $i<l_k$ (otherwise
$a_k>_{\omop^\X} a_{k+1}$ cannot hold) and we can let $f(n+1)=x_{k,i+1}
\cdot m_{k,i+1}$.

It is then straightforward
to obtain a $f$-computable, and hence $Z'$-computable, descending
sequence in $\X$.
\end{proof}

The proof above produces an index for a
$Z'$-computable descending subsequence in $\X$, uniformly in $\X$ and
the $Z$-computable descending sequence in $\omop^\X$.

\begin{corollary}
\ACA$\vdash\WO(\X\mapsto\omop^\X)$.
\end{corollary}
\begin{proof}
The previous proof can be formalized within \ACA.
\end{proof}

\begin{corollary}\label{cor forward ACApr}
\ACApr$\vdash\forall n\, \WO(\X\mapsto\iexpop n\X)$.
\end{corollary}
\begin{proof}
Theorem \ref{thm:om forward} implies that, given $n$, if $\X$ is a
$Z$-computable linear ordering, and $\iexpop n\X$ has a $Z$-computable
descending sequence, then $Z^{(n)}$ can compute a descending sequence
in $\X$. This can be formalized within \ACApr.
\end{proof}

The following two theorems are new in the form they are stated.
However, they can easily be obtained from the standard proof that
\ACA\ proves that every ordinal below $\varphi_2(0)$ can be proved
well-founded in \ACApl, and that every ordinal below $\Gamma_0$ can
be proved well-ordered in Predicative Analysis \cite{Fef64, Sch77}.

\begin{theorem}\label{thm:epsilon forward}
If $\X$ is a $Z$-computable linear ordering, and $\epsop_{\X}$ has a
$Z$-computable descending sequence, then $Z^{(\om)}$ can compute a
descending sequence in $\X$.
\end{theorem}
\begin{proof}
Let $(a_k:k\in\N)$ be a $Z$-computable descending sequence in
$\epsop_\X$. If no constant term $\eps_{x}$ appears in $a_0$, then
$a_0<\iexp{n_0}0$ for some $n_0$ so that we essentially have a
descending sequence in $\iexpop{n_0}0$. Then, applying $n_0$ times
Theorem \ref{thm:om forward}, we have that $Z^{(n_0)}$ computes a
descending sequence in $0$, a contradiction.

Thus we can let $x_0$ be the largest $x\in \X$ such that $\eps_{x}$
appears in $a_0$. It is not hard to prove by induction on terms that
$\eps_{x_0}\leq a_0<\iexp{n_0}{\eps_{x_0}+1}$ for some $n_0\in\N$. By
Lemma \ref{lemma:om eps}, ${\epsop_\X}\upto \iexp{n_0}{\eps_{x_0}+1}$
is computably isomorphic to $\iexpop{n_0}{\epsop_\X\upto
(\eps_{x_0}+1)}$ and we can view the $a_k$'s as elements of the latter.
Using Theorem \ref{thm:om forward} $n_0$ times, we obtain a
$Z^{(n_0)}$-computable descending sequence in $\epsop_\X\upto
(\eps_{x_0}+1)$. Noticing that the proof of Theorem \ref{thm:om
forward} is uniform, we can apply this process again to the sequence we
have obtained, and get an $x_1<_\X x_0$ and a descending sequence in
$\epsop_\X\upto (\eps_{x_1}+1)$ computable in $Z^{(n_0+n_1)}$ for some
$n_1\in \N$. Iterating this procedure we obtain a
$Z^{(\om)}$-computable descending sequence $x_0>_\X x_1>_\X \dots$ in
$\X$.
\end{proof}

\begin{corollary}\label{cor forward ACApl}
\ACApl $\vdash\WO(\X\mapsto\epsop_\X)$.
\end{corollary}
\begin{proof}
The previous proof can be formalized within \ACApl.
\end{proof}

\begin{theorem} \label{thm:veblen forward}
Let $\a$ be a $Z$-computable well-ordering. If $\X$ is a $Z$-computable
linear ordering, and $\phiop(\a,{\X})$ has a $Z$-computable descending
sequence, then $Z^{(\om^\a)}$ can compute a descending sequence in
$\X$.
\end{theorem}
\begin{proof}
By $Z$-computable transfinite recursion on $\a$, we define a computable
procedure that given a $Z$-computable index for a linear ordering $\X$
and for a descending sequence in $\phiop(\a,{\X})$, it returns a
$Z^{(\om^\a)}$-computable index for a descending sequence in $\X$. Let
$(a_k:k\in\N)$ be a computable descending sequence in
$\phiop(\a,{\X})$. Let $x_0$ be the largest $x\in \X$ such that the
constant term $\varphi_{\a,x}$ appears in $a_0$ (if no $\varphi_{\a,x}$
appears in $a_0$, just use $0$ in place of $\varphi_{\a,x_0}$ in the
argument below). It is not hard to prove by induction on terms that
$\varphi_{\a,x_0}\leq a_0<\varphi_{\b_0}^{n_0}(\varphi_{\a,x_0}+1)$ for
some $\b_0<\a$ and $n_0\in\N$, (where $\varphi^{n_0}_\b(z)$ is obtained
by applying the $\varphi_\b$ function symbol $n_0$ times to $z$). It
also not hard to show that $\phiop(\a,{\X}) \upto {\varphi_{\b_0}^{n_0}
(\varphi_{\a,x_0}+1)}$ is computably isomorphic to $\varphi^{n_0}(\b_0,
\phiop(\a,{\X\upto {x_0}})+1)$ (where $\phiop^{n_0}(\b,\Z)$ is obtained
by applying the $\phiop(\b,\cdot)$-operator on linear orderings $n_0$
times to $\Z$). Using the induction hypothesis $n_0$ times, we obtain a
$Z^{(\om^{\b_0}\cdot n_0)}$-computable descending sequence in
$\phiop(\a,{\X\upto {x_0}})+1$. Then, we apply this process again to
the sequence we have obtained, and get $x_1<_\X x_0$ and a descending
sequence in $\phiop(\a,{\X\upto {x_1}})+1$ computable in
$Z^{(\om^{\b_0} \cdot n_0+\om^{\b_1}\cdot n_1)}$ for some $\b_1<\a$ and
$n_1\in \N$. Iterating this procedure we obtain a $Z^{(\om^\a)}$
descending sequence $x_0>_\X x_1>_\X \dots$ in $\X$.
\end{proof}

\begin{corollary}\label{cor forward Pi0alpha}
Let $\a$ be a computable ordinal. Then
$\Pi^0_{\om^\a}$-\CA$\vdash\WO(\X\mapsto\varphi(\a,\X))$.
\end{corollary}
\begin{proof}
The previous proof can be formalized within $\Pi^0_{\om^\a}$-\CA\ for
a fixed computable $\a$.
\end{proof}

\begin{corollary}\label{cor forward ATR}
\ATR$\vdash\WO(\X \mapsto \phiop(\X,0))$.
\end{corollary}
\begin{proof}
Let $\a$ be a well-ordering and assume, towards a contradiction, that
there exists a descending sequence in $\phiop(\a,0)$. Let $Z$ be a real
such that both $\a$ and the descending sequence are $Z$-computable. By
Theorem \ref{thm:veblen forward} $Z^{(\om^\a)}$ (which exists in \ATR)
computes a descending sequence in $0$, which is absurd.
\end{proof}

%%%%%%%%%%%%%%%%%%%%%%%%%%%%%%%%%%%%%%%%%%%%%%%%%%%%%%%%%%%%%%%%%%%%%%%%%%%%%%%%%%%%%%%%%%%%%%%
\section{Ordinal exponentiation and the Turing Jump}\label{sect:exp}

In this section we give a proof of the second part of Theorem
\ref{Hirst}. Our proof is a slight modification of Hirst's proof, and
prepares the ground for the generalizations in the following sections.

We start by defining a modification of the Turing jump operator with nicer combinatorial properties.
We will then define two computable approximations to this jump operator, one from strings to strings, and the other one from trees to trees.

\begin{definition}\label{def:J}
Given  $Z \in \Bai$, we define the sequence of {\em $Z$-true stages} as follows:
\[
t_{n} = \max\{ t_{n-1}+1, \mu t(\{n\}^Z_t(n)\converges)\},
\]
starting with $t_{-1}=1$ (so that $t_n\geq n+2$). If there is no $t$
such that $\{n\}^Z_t(n)\converges$, then the above definition gives
$t_n=t_{n-1}+1$. So, $t_n$ is a stage where $Z$ can correctly guess
$Z'\upto n+1$ because $\forall m \leq n (m\in Z'\iff \{m\}^{Z\upto
t_n}(m) \converges)$. With this in mind, we define the \emph{Jump
operator} to be the function $\J\colon \Bai \to \Bai$ such that for
every $Z \in \Bai$ and $n \in \N$,
\[
\J(Z)(n) = Z\upto t_n,
\]
or equivalently
\[
\J(Z)=\la Z\upto t_0, Z\upto t_1, Z\upto t_2, Z\upto t_3, \dots\ra
\]
\end{definition}

Here is a sample of this definition: {\Small
\begin{eqnarray*}
\phantom{Z=\la Z(0), Z(1), Z(} t_0\phantom{), Z(3), Z(4), Z(5), Z(} t_1\phantom{),Z(} t_2\phantom{), Z(8), Z(9), Z(10), Z(11),Z(}t_3\phantom{),\cdots \ra}\\
Z= \la \underbrace{\underbrace{\underbrace{\underbrace{ Z(0), Z(1) }_{\J(Z)(0)}, Z(2), Z(3), Z(4), Z(5)}_{\J(Z)(1)}, Z(6)}_{\J(Z)(2)}, Z(7), Z(8), Z(9), Z(10), Z(11)}_{\J(Z)(3)}, Z(12),\cdots \ra
\end{eqnarray*}
}

Of course, $\J(Z) \teq Z'$ for every $Z$ as $n\in Z'\iff
\{n\}^{\J(Z)(n)}(n)\converges$. So, from a computability viewpoint,
there is no essential difference between $\J(Z)$ and the usual $Z'$.

\begin{definition}\label{def: J si}
The \emph{Jump function} is the mapping $J \colon \Seq \to \Seq$
defined as follows. For $\si \in \Seq$, define $t_{n} = \max\{
t_{n-1}+1, \mu t(\{n\}^{\si\upto t}(n)\converges)\},$ starting with
$t_{-1}=1$ (so that $t_n\geq n+2$). Again, if there is no $t$ such that
$\{n\}^{\si\upto t}(n)\converges$, then the above definition gives
$t_n=t_{n-1}+1$. Let $J(\si)=\la\si\upto t_0, \si\upto t_1, \dots,
\si\upto t_{k-1}\ra$ where $k$ is least such that $t_k>|\si|$.

Given $\tau\in J(\Seq)$, we let $K(\tau)$ be the last entry of $\tau$
when $\tau\neq\es$, and $K(\es)=\es$.
\end{definition}

\begin{remark}
Since we can computably decide whether $\{n\}^{\si\upto t}(n)
\converges$, the Jump function is computable. The computability of $K$
is obvious.
\end{remark}

The following Lemma lists the key properties of $J$ and $K$. We will
refer to these properties as (\ref{P1}), \dots, (\ref{P6}).

\begin{lemma} \label{lemma:J K}
For every $\si, \tau'\in \Seq$ and $\tau\in J(\Seq)$,
\begin{enumerate} \renewcommand{\theenumi}{P\arabic{enumi}}
\item $J(\si)=\es$ if and only if $|\si|\leq 1$. \label{P1}
\item $K(J(\si))= \si$ when $|\si|\geq 2$.     \label{P2}
\item $J(K(\tau))=\tau$.\label{P3}
\item If $\si\neq\si'$ and at least one has length $\geq 2$, then
    $J(\si)\neq J(\si')$.   \label{P4}
\item $|J(\si)|<|\si|$ and $|K(\tau)|>|\tau|$ except when
    $\tau=\es$.    \label{P5}
\item If $\tau'\subset\tau$ then $\tau'\in J(\Seq)$ and
    $K(\tau')\subset K(\tau)$.  \label{P6}
\end{enumerate}
\end{lemma}
\begin{proof}
(\ref{P1}) is obvious from the definition.

(\ref{P2}) follows from the fact that, when $|\si|\geq 2$, $t_{k-1} =
|\si|$ (using the notation of Definition \ref{def: J si}). In fact
$t_{k-1} \leq |\si|$ by definition of $k$, and if $t_{k-1}<|\si|$ then
we have either $\{k\}^{\si\upto t_k}(k)\converges$ (and hence $t_k \leq
|\si|$) or $t_k = t_{k-1}+1 \leq |\si|$, against the definition of $k$.

(\ref{P3}) follows from (\ref{P2}) and $K(\es)=\es$.

(\ref{P4}) follows immediately from (\ref{P1}) and (\ref{P2}).

The first part of (\ref{P5}) follows from $t_n\geq n+2$. The second
part is a consequence of the first, (\ref{P1}) and (\ref{P2}).

(\ref{P6}) is obvious when $\tau'=\es$, using the second part of
(\ref{P5}). Otherwise we have $\tau'=\la \si\upto t_0, \si\upto t_1,
\dots ,\si\upto t_j\ra$ for some $j<k-1$, so that $K(\tau')=\si\upto
t_j \subset \si\upto t_{k-1}=K(\tau)$. It is easy to check that
$\tau'=J(\si\upto t_j)$.
\end{proof}

The following Lemma explains how the Jump function approximates the
Jump operator.

\begin{lemma}\label{lemma:J approx J}
Given $Y,Z\in \Bai$, the following are equivalent:
\begin{enumerate}
\item $Y = \J(Z)$;
\item for every $n$ there exists $\si_n \subset Z$ with $|\si_n|>n$
    such that $Y\upto n =J(\si_n)$.
\end{enumerate}
\end{lemma}
\begin{proof}
Suppose first that $Y = \J(Z)$. When $n=0$ let $\si_n=Z\upto1$, which
works by (\ref{P1}). When $n>0$ let $\si_n = K(Y\upto n) = K(Y\upto n)
= \J(Z)(n-1) \subset Z$. If $\{0\}^Z (0)\converges$ then $Y(0) \subset
Z$ is such that $\{0\}^{Y(0)} (0)\converges$ and $Y(0) \subseteq \si_n$
so that also $\{0\}^{\si_n} (0)\converges$ and $J(\si_n)(0) = Y(0)$. If
$\{0\}^Z (0)\diverges$ then $Y(0) = Z\upto2 = \si_n\upto2 =
J(\si_n)(0)$. This is the base step of an induction that, using the
same argument, shows that $Y(i)=J(\si_n)(i)$ for every $i<n$. Thus $Y
\upto n \subseteq J(\si_n)$. By (\ref{P6}), we have $Y \upto n \in
J(\Seq)$ and we can apply (\ref{P3}) and (\ref{P5}) to obtain $Y \upto
n = J(\si_n)$ and $|\si_n|>n$.

Now assume that (2) holds, and suppose towards a contradiction that $Y
\neq \J(Z)$. Let $n$ be least such that $Y(n-1) \neq \J(Z)(n-1)$. If
$\si_n\subset Z$ is such that $Y\upto n = J(\si_n)$ we have
$J(\si_n)(n-1) \neq \J(Z)(n-1)$. This can occur only if
$\{n-1\}^{\si_n}(n-1) \diverges$ and $\{n-1\}^Z(n-1) \converges$, which
implies $n'>|\si_n|$, where $n'=|\J(Z)(n-1)|$. Notice that for any
$m>n'$ we have $J(Z\upto m)(n-1) = \J(Z)(n-1)$ and hence $J(Z\upto
m)(n-1) \neq Y(n-1)$. This contradicts the existence of
$\si_{n'}\subset Z$ with $|\si_{n'}|>n'$ such that $Y\upto {n'} =
J(\si_{n'})$.
\end{proof}

The following corollary is obtained by iterating the Lemma.

\begin{corollary}\label{cor:Jm approx JM}
For every $m>0$, given $Y,Z\in \Bai$, the following are equivalent:
\begin{enumerate}
\item $Y = \J^m(Z)$;
\item for every $n$ there exists $\si_n \subset Z$ with $|\si_n|
    \geq n+m$ such that $Y\upto n =J^m(\si_n)$.
\end{enumerate}
\end{corollary}

The Jump function leads to the definition of the Jump Tree.

\begin{definition}
Given a tree $T \subseteq \Seq$ we define the \emph{Jump Tree of
$T$} to be
\[
\JT(T)=\set{J(\si)}{\si\in T}.
\]
\end{definition}

The following lemmas summarize the main properties of the Jump Tree.

\begin{lemma}\label{lemma:JTcomp}
For every tree $T$, $\JT(T)$ is a tree computable in $T$.
\end{lemma}
\begin{proof}
$\JT(T)$ is a tree because if $\tau \subset J(\si)$ for $\si\in T$,
then $\tau=J(K(\tau))$ (by (\ref{P6}) and (\ref{P3})) and $K(\tau)\in
T$ (since by (\ref{P6}), (\ref{P2}) and (\ref{P1}), $K(\tau) \subset
K(J(\si)) \subseteq \si$).

$\JT(T)$ is computable in $T$ because $\tau\in \JT(T)$ if and only if
$\tau=J(K(\tau))$ (which is equivalent to $\tau \in J(\Seq)$ by
(\ref{P3})) and $K(\tau)\in T$.
\end{proof}

\begin{lemma}\label{lemma:JT paths}
For every tree $T$, $[\JT(T)] = \set{\J(Z)}{Z \in [T]}$.
\end{lemma}
\begin{proof}
First let $Z \in [T]$. Since by Lemma \ref{lemma:J approx J} for every
$n \in \N$, $\J(Z) \upto n = J(\si)$ for some $\si \subset Z$, so
$\J(Z) \upto n \in \JT(T)$. This implies $\set{\J(Z)}{Z \in [T]}
\subseteq [\JT(T)]$.

To prove the other inclusion, fix $Y \in [\JT(T)]$, notice that $Y(n)
\subset Y(n+1)\in\Seq$ for every $n$, and let $Z
=\bigcup_{n\in\N}Y(n)\in \Bai$. Observe that, again by Lemma
\ref{lemma:J approx J}, $Y = \J(Z)$ and $Z\in [T]$.
\end{proof}

We can now define the $Z$-computable linear ordering of theorem
\ref{Hirst}: let $\X_Z = \la \JT(T_Z), {\KB} \ra$ where $T_Z = \set{Z
\upto n}{n \in \N}$. Note that $\X_Z$ is indeed a linear ordering and,
by Lemma \ref{lemma:JTcomp}, it is $Z$-computable. Since $Z$ is the
unique path in $T_Z$, by Lemma \ref{lemma:JT paths} $\J(Z)$ is the
unique path in $\JT(T_Z)$. Moreover, for every $\tau = J(\si) \in
\JT(T_Z)$ we have that either $\tau \subset \J(Z)$ or there is some $i$
such that $\tau \upto i = \J(Z) \upto i$ and $\tau(i) \neq \J(Z)(i)$.
This can only happen if $\{i\}^\si(i) \diverges$ and $\{i\}^Z(i)
\converges$, so that $\tau(i) \subset \J(Z)(i)$. By our assumption on
the coding of strings, we have $\tau(i) < \J(Z)(i)$ and hence $\tau
\KBst \J(Z) \upto |\tau|$.

Let $\la \tau_n \ra_{n \in \N}$ be an infinite \KBst-descending
sequence in $\JT(T_Z)$. If $\tau_n \not\subset \J(Z)$ for some $n$ then
$\tau_m \KBst \tau_n \KBst \J(Z)\upto|\tau_m|$ for all $m>n$, which by
Lemma \ref{lemma:KB} implies the existence of a path in $\JT(T_Z)$
different from $\J(Z)$, a contradiction. Therefore any infinite
descending sequence in $\X_Z$ consists only of initial segments of
$\J(Z)$ and hence computes $\J(Z) \teq Z'$.

We still need to prove the existence of a $Z$-computable descending
sequence in $\omop^{\X_Z}$. To this end we use of the following
function.

\begin{definition}
Let $T$ be a tree and order $\JT(T)$ by \KB. Define $h\colon T \to
\omop^{\JT(T)}$ by
\[
h(\si) = \left( \sum_{\substack{i<|J(\si)|\\ \{i\}^\si(i)\diverges}}
\om^{J(\si) \upto i} \right) + \om^{J(\si)} \cdot 2
\]
for $\si\neq \es$ and $h(\es) = \om^{\es}\cdot 3$.
\end{definition}

The sum above is written in \KB-decreasing order, so that indeed
$h(\si) \in \omop^{\JT(T)}$.

Since $J$ is computable, $h$ is computable as well.

The proof below should help the reader understand the motivation for the definition above.

\begin{lemma}\label{h:monotone}
$h$ is $({\supset}, {<_{\omop^{\JT(T)}}})$-monotone.
\end{lemma}

\begin{proof}
Suppose $\rho, \si\in T$ are such that $\rho \supset \si$; we want to
show that $h(\rho) <_{\omop^{\JT(T)}} h(\si)$.

If $\si=\es$ then $\om^\es$ occurs with multiplicity $3$ in $J(\si)$
and with multiplicity at most $2$ in $J(\rho)$. Since $\es$ is the
\KB-maximum element in \Seq (and hence also in $\JT(T)$), this implies
$h(\rho) <_{\omop^{\JT(T)}} h(\si)$.

If $\si\neq\es$ then $J(\si)\neq J(\rho)$ by (\ref{P4}). Since
$\si\subset\rho$, if $\{i\}^\rho(i) \diverges$ then $\{i\}^\si(i)
\diverges$ as well. Thus there are two possibilities. If for all
$i<|J(\si)|$, $\{i\}^\si(i) \diverges$ whenever $\{i\}^\rho(i)
\diverges$ then $J(\si) \subset J(\rho)$ and the first difference
between $h(\si)$ and $h(\rho)$ is the coefficient of $\om^{J(\si)}$,
which in $h(\si)$ is $2$ and in $h(\rho)$ is either $1$ or $0$
(depending on whether $\{|J(\si)|\}^\rho(|J(\si)|) \diverges$ or not).
In any case, $h(\rho) <_{\omop^{\JT(T)}} h(\si)$. If instead for some
$i<|J(\si)|$, $\{i\}^\si(i) \diverges$ and $\{i\}^\rho(i) \converges$
let $i_0$ be the least such $i$. Then the first difference between
$h(\si)$ and $h(\rho)$ occurs at $\om^{J(\si \upto i_0)}$, which
appears in $h(\si)$ but not in $h(\rho)$. Again, we have $h(\rho)
<_{\omop^{\JT(T)}} h(\si)$.
\end{proof}

We can now finish off the proof of the second part of Theorem
\ref{Hirst}. The sequence $\la h (Z \upto n) \ra_{n \in \N}$ is
$Z$-computable and strictly decreasing in $\omop^{\X_Z}$ by Lemma
\ref{h:monotone}.

Obviously our proof yields the following generalization of the second
part of Theorem \ref{Hirst}.

\begin{theorem}\label{Hirst:gen}
For every real $Z$ there exists a $Z$-computable linear ordering $\X$
with a $Z$-computable descending sequence in $\omop^{\X}$ such that
every descending sequence in $\X$ computes $Z'$.
\end{theorem}

%%%%%%%%%%%%%%%%%%%%%%%%%%%%%%%%%%%%%%%%%%%%%%%%%%%%%%%%%%%%%%%%%%%%%%%%%%%%%%%%%%%%%%%%%%%%%%%%%%
\section{The $\eps$ function and the $\om$-Jump}\label{sect:eps}
In this section we extend the construction of Section \ref{sect:exp}.
To iterate the construction, even only a finite number of times, requires
generalizing the definition of $h$. Then we tackle the issue of
extending the definition at limit ordinals by considering the
$\om$-Jump.

\subsection{Finite iterations of exponentiation and Turing Jump}\label{ssect:finite}
We start by defining a version of the function $h$ used in the previous
section that we can iterate.

\begin{definition}\label{def:hgJ}
Let \X\ be a linear ordering, $T$ a tree and
\[
g\colon \JT(T) \to \X
\]
 a
function. Define
\[
h_g \colon T \to \omop^\X
\]
 by
\[
h_g(\si) = \left( \sum_{\substack{i<|J(\si)|\\ \{i\}^\si(i)\diverges}}
\om^{g(J(\si) \upto i)} \right) + \om^{g(J(\si))} \cdot 2
\]
for $\si\neq \es$ and $h_g(\es) = \om^{g(\es)}\cdot 3$.
\end{definition}

Note that when $g$ is the identity, then $h_g = h$ of the previous
section. Also, $h_g$ is $g$-computable.

\begin{lemma}\label{lemma:monotone}
If $g$ is $({\supset}, {<_\X})$-monotone, then $h_g$ is $({\supset},
{<_{\omop^\X}})$-monotone.
\end{lemma}
\begin{proof}
Notice that $g$ $({\supset}, {<_\X})$-monotone implies that the sum in
the definition of $h_g(\si)$ is written in decreasing order. The proof
is the same as the one for Lemma \ref{h:monotone}.
\end{proof}

We can now prove the analogue of Theorem \ref{Hirst} for iterations of
the exponential (recall the notation $\iexpop n\X$ introduced in Definition
\ref{iexpop}).

\begin{theorem}\label{0^n}
For every $n \in \N$ and $Z \in \Bai$, there is a $Z$-computable linear
ordering $\X_Z^n$ such that the jump of every descending sequence in
$\X_Z^n$ computes $Z^{(n)}$, but there is a $Z$-computable descending
sequence in $\iexpop n{\X_Z^n}$.
\end{theorem}
\begin{proof}
Letting again $T_Z = \set{Z \upto n}{n \in \N}$, we define a sequence
$\la T_i \ra_{i \leq n}$ of trees as follows: let $T_0 = T_Z$ and
$T_{i+1} = \JT(T_i)$ for every $i<n$. By induction on $i$, using Lemmas
\ref{lemma:JTcomp} and \ref{lemma:JT paths}, we have that each $T_i$ is
a $Z$-computable tree and that the only path through $T_i$ is $\J^i
(Z)$ (i.e.\ the result of applying $i$ times $\J$ starting with $Z$).
We let $\X_Z^n = \la T_n, {\KB} \ra$, which is a $Z$-computable linear
ordering. By Lemma \ref{lemma:KB} if $f$ is a descending sequence in
$\X_Z^n$ then $\J^n (Z) \leqt f'$. Since $Z^{(n)} \teq \J^n (Z)$, the
first property of $\X_Z^n$ is proved.

To show that there is a $Z$-computable descending sequence in $\iexpop
n{\X_Z^n}$ we define by recursion on $m \leq n$ functions $g_m\colon
T_{n-m} \to \iexpop m{\X_Z^n}$. Let $g_0\colon T_n \to \X_Z^n$ be the
identity function ($T_n$ is indeed the domain of $\X_Z^n$). We define
$g_{m+1} \colon T_{n-m-1} \to \iexpop {m+1}{\X_Z^n}$ by $g_{m+1} =
h_{g_m}$ as in Definition \ref{def:hgJ}. By induction on $m \leq n$,
using Lemma \ref{lemma:monotone}, it is immediate that each $g_m$ is
$({\supset}, {<_{\iexpop m{T_n}}})$-monotone and computable. Hence the
sequence $\la g_n (Z \upto j) \ra_{j \in \N}$ in $\iexpop n{\X_Z^n}$ is
$Z$-computable and descending.
\end{proof}

%%%%%%%%%%%%%%%%%%%%%%%%%%%%%%%%%%%%%%
\subsection{The $\om$-Jump}\label{ssect:eps}

Now we define the iteration of the Jump operator at the first limit
ordinal $\om$. Again, our definition is slightly different than the usual one
so that it has nicer combinatorial properties.
The difference being that instead of pasting all the $\J^i(Z)$ together as columns, we will take only the first value of each.
Later we will show that this is enough.
We will also define two
computable approximations to this $\om$-jump operator, one from strings
to strings, and the other one from trees to trees, and a computable
inverse function.

\begin{definition}
We define the \emph{$\om$-Jump operator} to be the function
$\J^\om\colon \Bai \to \Bai$ such that for every $Z \in \Bai$
\[
\J^\om(Z) = \la \J(Z)(0),\  \J^2(Z)(0),\  \J^3(Z)(0),\dots\ra,
\]
or, in other words, $\J^\om(Z)(n)=\J^{n+1}(Z)(0)$.
\end{definition}

Notice that $\J^\om(\J(Z))$ equals $\J^\om(Z)$ with the first element
removed. Before showing that $\J^\om(Z) \teq Z^{(\om)}$ it is
convenient to define the inverse of $\J^\om$.

\begin{definition}
Given $Y\in \J^\om(\Bai)$ we define
\[
\K^{\om}(Y) = \bigcup_{n} K^{n}(Y(n)).
\]
\end{definition}

We need to show that the union above makes sense. Assume $Y=\J^\om(Z)$.
It might help to look at Figure \ref{fig: Y=Jom Z}. Notice that for
each $n$, $Y(n) \subset \J^n(Z)$ because for every $X$, $\J(X)(0)
\subset X$ and $Y(n)=\J(\J^n(Z))(0)$. We also know that if $\si \subset
\J(X)$, then $K(\si)\subset X$, so that $K^{n}(Y(n))\subset Z$. It
follows that $\bigcup_{n} K^{n}(Y(n))\subseteq Z$. Applying (\ref{P5})
$n$ times we get that $|K^{n}(Y(n))| > n$, and therefore the union
above does actually produce $Z$. We have just proved the following
lemma.

\begin{figure}
{\Small
\xymatrix@C=12pt@R=12pt{
\J^\om(Z)  \ar@{}|(.6){=}[r]  \ar@{}|(.4){\shortparallel}[d] \ar@{}|(.8){\frown}[d] & Y\\
\vdots     &   \vdots \\
\J^4(Z)(0)   \ar@{}|(.6){=}[r]    & Y(3) \ar@{}|{\subset}[r] & \cdots    & &   & &\vdots  \\
\J^3(Z)(0)   \ar@{}|(.6){=}[r]    & Y(2)  \ar@{}|(.4){\subset}[r] &  K(Y(3)) \ar_{J}[ul]  \ar@{}|{\subset}[r] & \cdots& \cdots& \cdots \ar@{}|(.3){\subset}[r]     & \J^2(Z)	\\
\J^2(Z)(0)   \ar@{}|(.6){=}[r]    & Y(1)  \ar@{}|(.4){\subset}[r] &  K(Y(2)) \ar_{J}[ul]  \ar@{}|{\subset}[r]	&  K^2(Y(3)) \ar_{J}[ul] \ar@{}|{\subset}[r] & \cdots& \cdots \ar@{}|(.3){\subset}[r]     & \J(Z) \ar_{\J}[u]\\
\J(Z)(0)   \ar@{}|(.6){=}[r]    & Y(0)  \ar@{}|(.4){\subset}[r] &  K(Y(1))  \ar_{J}[ul]  \ar@{}|{\subset}[r]	&  K^2(Y(2))  \ar_{J}[ul]\ar@{}|{\subset}[r]&    K^3(Y(3))   \ar_{J}[ul] \ar@{}|(.7){\subset}[r] & \cdots\ar@{}|(.3){\subset}[r]   & Z \ar_{\J}[u]   \\
\ar@{}|{\smile}[u] &   K^\om(Y\upto 1) \ar@{}|{\subset}[r]\ar@{}|{\shortparallel}[u]   &   K^\om(Y\upto 2) \ar@{}|{\subset}[r]\ar@{}|{\shortparallel}[u]   &   K^\om(Y\upto 3) \ar@{}|{\subset}[r]\ar@{}|{\shortparallel}[u]   &   K^\om(Y\upto 4) \ar@{}|(.7){\subset}[r]\ar@{}|{\shortparallel}[u]  & \cdots\ar@{}|(.3){\subset}[r]    &\K^\om(Y) \ar@{}|{\shortparallel}[u] \\
}}
\caption{Assuming $Y=\J^\om(Z)$.}   \label{fig: Y=Jom Z}
\end{figure}

\begin{lemma}\label{lemma:Kom Bai}
For every $Z\in \Bai$, $\K^\om(\J^\om(Z))=Z$.
\end{lemma}

\begin{lemma} \label{lemma: Z om jump}
For every $Z\in \Bai$, $\J^\om(Z) \teq Z^{(\om)}$.
\end{lemma}
\begin{proof}
We already know that $Z^{(n)}\teq \J^n(Z)$ uniformly in $n$ and hence
that $Z^{(\om)} = \bigoplus_{n\in\N}Z^{(n)}  \teq \bigoplus_{n\in\N}
\J^n(Z)$. It immediately  follows that  $\J^\om(Z) \leqt Z^{(\om)}$.
For the other direction we need to uniformly compute all the reals
$\J^n(Z)$ from $\J^\om(Z)$. We do this as follows. Given $Y\in\Bai$,
let $Y^{-n}$ be $Y$ with its first $n$ elements removed. Then,
$\J^\om(\J^n(Z))=\J^\om(Z)^{-n}$. By the lemma above we get that
$\J^n(Z)=\K^\om(\J^\om(Z)^{-n})$, which we can compute uniformly from
$\J^\om(Z)$.
\end{proof}

As in section \ref{sect:exp}, where we computably approximated the jump
operator, we will now approximate the $\om$-Jump operator with a
computable operation on finite strings.

\begin{definition}\label{def:Jom si}
The \emph{$\om$-Jump function} is the map $J^\om \colon \Seq \to \Seq$
defined as follows. Given $\si\in \Seq$, let
\[
J^\om(\si) = \la J(\si)(0),\  J^2(\si)(0),\dots,  J^{n-1}(\si)(0)\ra,
\]
where $n$ is the least such that $J^n(\si)=\es$ (there is always such
an $n$, because, by (\ref{P5}), $|J^i(\si)|\leq|\si|-i$ for $i\leq |\si|$). Note that then, by (\ref{P1}),
$|J^{n-1}(\si)|=1$.
\end{definition}

$J^\om$ is computable (because $J$ is computable) and we will now
define its computable partial inverse $K^\om$.

\begin{definition}\label{def:Kom si}
Given $\tau\in J^\om(\Seq)$, let $K^\om(\tau)= K^{|\tau|}(\ell(\tau))$
(recall that $\ell(\tau)=\la \tau(|\tau|-1|)\ra$) when $\tau\neq\es$, and
$K^\om(\es)=\es$.
\end{definition}

The following properties are the analogues of those of Lemma
\ref{lemma:J K} for the $\om$-Jump function and its inverse. We will
refer to them as (\ref{Pom1}), \dots, (\ref{Pom7}).

\begin{lemma}  \label{lemma:Jom Kom}
For $\si,\si', \tau'\in \Seq$, $\tau\in J^\om(\Seq)$,
\begin{enumerate}\renewcommand{\theenumi}{P$^\om$\arabic{enumi}}
\item $J^\om(\si)=\es$ if and only if $|\si|\leq 1$. \label{Pom1}
\item $K^\om(J^\om(\si))= \si$ for $|\si|\geq 2$. \label{Pom2}
\item $J^\om(K^\om(\tau))=\tau$.        \label{Pom3}
\item If $\si\neq\si'$ and at least one has length $\geq 2$, then
    $J^\om(\si)\neq J^\om(\si')$.   \label{Pom4}
\item $|J^\om(\si)|<|\si|$ and $|K^\om(\tau)|>|\tau|$ except when
    $\tau=\es$.    \label{Pom5}
\item If $\tau'\subset\tau$ then $\tau'\in J^\om(\Seq)$ and
    $K^\om(\tau')\subset K^\om(\tau)$.      \label{Pom6}
\item If $J^\om(\si')\subseteq J^\om(\si)$ then, for every $m$, $J^m(\si')\subseteq J^m(\si)$.   \label{Pom7}
\end{enumerate}
\end{lemma}
\begin{proof}
(\ref{Pom1}) follows from (\ref{P1}) and the fact that $J^\om(\si)=\es$
is equivalent to $J(\si)=\es$.

To prove (\ref{Pom2}) let $|J^\om(\si)|=n>0$. Then $\ell(J^\om(\si)) =
\la J^n(\si)(0) \ra = J^n(\si)$ because $|J^n(\si)|=1$ as noticed in
the definition of $J^\om$. Thus $K^\om(J^\om(\si)) = K^n(J^n(\si)) =
\si$ by (\ref{P2}).

(\ref{Pom3}) follows from (\ref{Pom2}) and $K(\es)=\es$. (\ref{Pom4})
follows immediately from (\ref{Pom1}) and (\ref{Pom2}). The first part
of (\ref{Pom5}) is immediate because by (\ref{P5}) we have
$J^{|\si|}(\si)=\es$. The second part of (\ref{Pom5}) is a consequence
of the first, (\ref{Pom1}) and (\ref{Pom2}).

For (\ref{Pom6}) look at Figure \ref {fig tau = Jom si}.
\begin{figure}
{\Small
\[
\xymatrix@=10pt{
\tau \ar@{}|{=}[r] \ar@{}|(.4){\shortparallel}[dd]& \J^{\om}(\si)  \ar@{}|(.4){\shortparallel}[dd] \\
\ar@{}|{\frown}[d]	&	\ar@{}|{\frown}[d] &\ar@{-}[dddddd]	&&   \emptyset \\
\tau(|\tau|-1)\ar@{}|{=}[r]	&J^{|\tau|}(\si)(0)&	&\la \tau(|\tau|-1)\ra\ar@{}|(.6){=}[r]&   J^{|\tau|}(\si)  \ar_J[u] \\
\vdots	& \vdots			&&&&	\ddots \ar_J[ul] \\
\tau(|\tau'|-1)\ar@{}|{=}[r]	& J^{|\tau'|}(\si)(0)			&& \la \tau(|\tau'|-1)\ra\ar@{}|(.6){=}[r]  &J^{|\tau'|}(\si') \ar@{}|(.8){\subset}[r]  \ar_K[dr]  & \cdots  \ar@{}|(.3){\subset}[r] &  J^{|\tau'|}(\si)  \ar_J[ul] \\
\vdots& \vdots	&&&&\ddots \ar_K[dr]&		&\ddots  \ar_J[ul]\\
\tau(0)\ar@{}|{=}[r]& J(\si)(0)	&&&&& J(\si') \ar_K[dr]  \ar@{}|(.6){\subset}[r]    & \cdots  \ar@{}|(.3){\subset}[r]    & J(\si) \ar_J[ul]\\
\ar@{}|{\smile}[u]	&	\ar@{}|{\smile}[u]&&&&&		 &\si'	\ar@{}|(.6){\subset}[r]    & \cdots  \ar@{}|(.3){\subset}[r]       &  \si \ar_J[ul]\\
}
\]}
\caption{Assuming $\tau=J^\om(\si)$ and $\tau'\subset\tau$.}\label{fig tau = Jom si}
\end{figure}
(\ref{Pom6}) is obvious when $\tau'=\es$, using the second part of
(\ref{Pom5}). Otherwise, let $\si$ be such that $\tau=J^\om(\si)$. The
idea is to define $\si'\subset \si$ as in the picture and then show
that $\tau'=J^\om(\si')$. Notice that $|\si|>|\tau|\geq2$ by
(\ref{Pom5}), and that $\si=K^\om(\tau)$ by (\ref{Pom2}). Notice also
that
\[
\ell(\tau') = \la \tau(|\tau'|-1)\ra = \la J^{|\tau'|}(\si)(0) \ra \subset J^{|\tau'|}(\si)
\]
because $|\tau'|<|\tau|$ and hence $|J^{|\tau'|}(\si)|>1$. Let $\si' =
K^{|\tau'|}(\ell(\tau'))$. Using (\ref{P6}) $|\tau'|$ times we know
that $\si' \subset \si$ and $J^{|\tau'|}(\si')=\ell(\tau')$. Now, we
need to show that $J^{\om}(\si')=\tau'$. First notice that
$|J^\om(\si')| = |\tau'|$ because $|J^{|\tau'|} (\si')| = |\ell(\tau')|
=1$. By induction on $i \leq |\tau'|$ we can show, using (\ref{P6}) and
(\ref{P2}), that
\begin{equation}\label{aa}
J^{|\tau'|-i}(\si') = K^i(\ell(\tau')) \subset K^i(J^{|\tau'|}(\si)) = J^{|\tau'|-i}(\si).
\end{equation}

Now for $j<|\tau'|$, $J^\om(\si')(j) = J^{j+1}(\si')(0) = J^{j+1}(\si)(0) = \tau(j) = \tau'(j)$.

For (\ref{Pom7}) let $\tau'=J^\om(\si')$. Then, if $i=|\tau'|-m$,
equation (\ref{aa}) shows that $J^m(\si') \subseteq J^m(\si)$.
\end{proof}

As we did in Section \ref{sect:exp}, we now explain how the $\om$-Jump
function approximates the $\om$-Jump operator.

\begin{lemma}\label{lemma:Jom approx Jom}
Given $Y,Z\in \Bai$, the following are equivalent:
\begin{enumerate}
\item $Y = \J^\om(Z)$;
\item for every $n$ there exists $\si_n\subset Z$ with $|\si_n|>n$    such that $Y\upto n =J^\om(\si_n)$.
\end{enumerate}\end{lemma}
\begin{proof}
First assume $Y = \J^\om(Z)$. When $n=0$ let $\si_0=Z\upto1$, which
works by (\ref{Pom1}). When $n>0$ let $\si_n = K^\om(Y \upto n)$. We
recommend the reader to look at Figure \ref{fig: Y=Jom Z} again. By
(\ref{Pom3}) we have $Y\upto n =J^\om(\si_n)$. Since $\si_n = K^n(\la
Y(n-1)\ra) = K^{n-1}(Y(n-1))$, $\si_n$ is one of the strings occurring
in the definition of $\K^\om(Y)$ and hence $\si_n \subset \K^\om(Y)=Z$
by Lemma \ref{lemma:Kom Bai}. We get that $|\si_n|>n$ by (\ref{P5})
applied $n$ times to $\si_n=K^n(\la Y(n-1)\ra)$.

Suppose now that (2) holds. By (\ref{Pom7}), we get that for all $m$
and $n$, $J^m(\si_n)\subseteq J^m(\si_{n+1})$. Using Corollary
\ref{cor:Jm approx JM}, it is straightforward to show that for all
$m<n$, $\bigcup_nJ^m(\si_n) = \J^m(Z)$ and hence
$J^m(\si_n)(0)=\J^m(Z)(0)$. It follows that for every $m$ and $n>m$
\[
\J^\om(Z)(m)=\J^{m+1}(Z)(0)=J^{m+1}(\si_{n})(0) = J^\om(\si_{n})(m)= Y(m).\qedhere
\]
\end{proof}

Again as in Section \ref{sect:exp}, the $\om$-Jump function leads to
the definition of the $\om$-Jump Tree.

\begin{definition}
Given a tree $T \subseteq \Seq$ the \emph{$\om$-Jump Tree of $T$}
is
\[
\JT^\om(T) = \set{J^\om(\si)}{\si\in T}.
\]
\end{definition}

\begin{lemma}\label{lemma:JTom comp}
For every tree $T$, $\JT^\om(T)$ is a tree computable in $T$.
\end{lemma}
\begin{proof}
The proof is identical to the proof of Lemma \ref{lemma:JTcomp}, using
Lemma \ref{lemma:Jom Kom} in place of Lemma \ref{lemma:J K}.
\end{proof}

\begin{lemma}\label{lemma:JTom paths}
For every tree $T$, $[\JT^\om(T)] = \set{\J^\om(Z)}{Z \in [T]}$.
\end{lemma}
\begin{proof}
To prove $\set{\J^\om(Z)}{Z \in [T]} \subseteq [\JT^\om(T)]$ we can
argue as in the proof of Lemma \ref{lemma:JT paths}, using Lemma
\ref{lemma:Jom approx Jom} in place of Lemma \ref{lemma:J approx J}.

To prove the other inclusion, fix $Y \in [\JT^\om(T)]$. For each $n$,
let $\si_n=K^\om(Y\upto n)\in T$. Since $J^\om(\si_n) = Y \upto n$ by
(\ref{Pom3}), we have $\si_n \subset \si_{n+1}$ for each $n$ by
(\ref{Pom6}). Let $Z=\bigcup_{n\in \N}\si_n \in [T]$. Then, by Lemma
\ref {lemma:Jom approx Jom} we get $Y=\J^\om(Z)$.
\end{proof}

\subsection{$\om$-Jumps versus the epsilon function}\label{ssect:omJ
eps}

Our goal now is to generalize Definition \ref{def:hgJ} with an operator
that uses $\epsop$ rather than $\omop$. We thus wish to define an
operator $h^\om$ that, given an order preserving function $g\colon
\JT^\om (T) \to \X$ (where \X\ is a linear order), returns an order
preserving function $h^\om_g \colon T \to \epsop_\X$. To do so we will
iterate the $h$ operator of Definition \ref{def:hgJ} along the elements
of $\JT^\om(T)$.

Let us give the rough motivation behind the definition of the operator
$h^\om$ below. Suppose we are given an order preserving function
$g\colon\JT^\om (T) \to \X$. For each $i$, we would like to define a
monotone function $f_i\colon \JT^i(T)\to \epsop_\X$ such that
$f_i=h_{f_{i+1}}$, where $h_{f_{i+1}}$ is as in Definition
\ref{def:hgJ}. Notice that the range of this function is correct, using
the fact that $\omop^{\epsop_\X}$ is computably isomorphic to
$\epsop_\X$. However, we do not have a place to start as to define such
$f_i$ we would need $f_{i+1}$, and this recursion goes the wrong way.
Note that if $\tau=J^\om(\si)\in \JT^\om (T)$, then $\la\tau(i)\ra \in
\JT^i(T)$, and we could use $g$ to define $f_i$ at least on the strings
of length 1, of the form $\la\tau(i)\ra$. (This is not exactly what we
are going to do, but it should help picture the construction.) The good
news is that to calculate $f_{i-1} = h_{f_i}$ on strings of length at
most 2, we only need to know the values of $f_i$ on strings of length
at most 1. Inductively, this would allow us to calculate $f_0\colon
T\to \X$ on strings of length at most $i$. Since this would work for
all $i$, we get $f_0$ defined on all $T$. We now give the precise
definition.

First, we need to iterate the Jump Tree operator along any finite
string.

\begin{definition}
If $T$ is a tree we define
\[
\JT^\om_\tau(T) = \set{J^{|\tau|+1}(\si)}{\si\in T \land \tau\subseteq J^{\om}(\si)}.
\]
\end{definition}

Notice that $\JT^\om_\tau(T) \subseteq \JT^{|\tau|+1}(T)$ and that
$\JT^\om_\tau(T)$ is empty when $\tau \notin \JT^\om(T)$. The following
Lemma provides an alternative way of defining $\JT^\om_\tau(T)$ by an
inductive definition.

\begin{lemma}\label{lemma: JT tau recursive}
Given a tree $T \subseteq \Seq$,
\begin{align*}
\JT^\om_\es(T)   &=    \JT(T)  \\
\JT^\om_{\tau\conc\la c\ra}  (T)  &=   \JT(\JT^\om_\tau(T)_{\la c\ra}).
\end{align*}
($T_{\la c\ra}$ was defined in \ref{def:Tsigma} as $\set{\rho\in T}{\la
c\ra \subseteq \rho \lor \rho=\es}$.)
\end{lemma}
\begin{proof}
Straightforward induction on $|\tau|$.
\end{proof}

The next Lemma links $\JT^\om_\tau(T)$ to $\JT^\om(T)$.

\begin{lemma}
Given a tree $T \subseteq \Seq$, $\tau\in \Seq$, and $c \in \N$,
\[
\tau\conc \la c\ra  \in \JT^\om(T)       \iff      \la c \ra \in \JT^\om_{\tau}(T).
\]
\end{lemma}
\begin{proof}
This follows immediately from the definitions of $\JT^\om(T)$ and
$\JT^\om_{\tau}(T)$.
\end{proof}

\begin{definition}\label{def:fgtau}
Let \X\ be a linear ordering and $g\colon \JT^\om (T) \to \X$ be a
function. We define simultaneously for each $\tau \in \JT^\om(T)$ a
function
\[
f_\tau \colon \JT^\om_\tau (T) \to \epsop_\X
\]
by recursion on $|\si|$:
\[
f_\tau (\si) =
\begin{cases}
\eps_{g(\tau)} & \text{if $\si = \es$;}\\
h_{f_{\tau \conc \la \si(0) \ra}} (\si) & \text{if $\si \neq \es$.}
\end{cases}
\]
Here $h_{f_{\tau \conc \la \si(0) \ra}}$ is defined according to
Definition \ref{def:hgJ}. We then define
\[
h^\om_g = h_{f_\es}\colon T \to \epsop_\X.
\]
\end{definition}

\begin{remark}
First of all notice that we are really doing a recursion on $|\si|$. In
fact, to compute $h_{f_{\tau \conc \la \si(0) \ra}} (\si)$ when $\si
\neq \es$ we use $f_{\tau \conc \la \si(0) \ra}$ on strings of the form
$J(\si) \upto i$, which have length $\leq |J(\si)|<|\si|$ by
(\ref{P5}).

Let us notice the functions have the right domains and ranges. The
proof is done simultaneously for all $\tau\in \JT^\om(T)$ by induction
on $|\si|$. Take $\si\in \JT^\om_\tau(T)$ with $|\si|=n$. Suppose that
for all $\tau'\in \JT^\om(T)$ and all $\si'\in \JT^\om_{\tau'}(T)$ with
$|\si'|<n$ we have that $f_{\tau'}(\si')$ is defined and
$f_{\tau'}(\si') \in \epsop_\X$.

If $\si=\es$, then $f_\tau(\si) = \eps_{g(\tau)} \in \epsop_\X$.
Suppose $\si\neq\es$ and let $\tau'=\tau\conc\la\si(0)\ra$. Then
$f_\tau(\si)=h_{f_{\tau'}}(\si)$. When computing $h_{f_{\tau'}}(\si)$,
we only apply $f_{\tau'}$ to strings of the form $J(\si) \upto i$.
These strings have length less than $n$ and, by Lemma
\ref{lemma:JTcomp}, belong to $\JT(\JT^\om_\tau(T)_{\la\si(0)\ra})=
\JT^\om_{\tau'}(T)$ (by Lemma \ref{lemma: JT tau recursive}). By the
induction hypothesis we have that $f_{\tau'}$ is defined on these
strings and takes values in $\epsop_\X$. Therefore,
$h_{f_{\tau'}}(\si)$ is defined and $h_{f_{\tau'}}(\si) \in
\omop^{\epsop_\X}$. Using that $\omop^{\epsop_\X}= \epsop_\X$, we get
that $f_\tau \colon \JT^\om_\tau (T) \to \epsop_\X$.

Finally, since $f_\es \colon \JT(T) \to \epsop_\X$, we get that
$h^\om_g \colon T \to \epsop_\X$.
\end{remark}

\begin{lemma}\label{lemma:om monotone}
If $g\colon \JT^\om(T)\to\X$ is $(\supset,<_\X)$-monotone, then
$h_g^\om\colon T\to\epsop_\X$ is $(\supset,<_{\epsop_\X})$-monotone.
\end{lemma}
\begin{proof}
First, we note that by Lemma \ref{lemma:monotone}, it suffices to show
that $f_\es$ is $(\supset,<_{\epsop_\X})$-monotone. We will actually
show that for every $\tau\in \JT^\om(T)$, $f_\tau$ is
$(\supset,<_{\epsop_\X})$-monotone.

The proof is again done simultaneously for all $\tau\in \JT^\om(T)$ by
induction on the length of the strings. Suppose that on strings of
length less than $n$, for every $\tau'$, $f_{\tau'}$ is
$(\supset,<_{\epsop_\X})$-monotone. Let $\si' \subset \si \in
\JT^\om_\tau (T)$ with $|\si|=n$. Let  $\tau'=\tau\conc\la\si(0)\ra$.
Consider first the case when $\si'=\es$. Then $f_\tau(\si') =
\eps_{g(\tau)}$ while $f_\tau(\si)$ is a finite sum of terms of the
form $\om^{f_{\tau'} (J(\si)\upto i)}$. By the induction hypothesis,
the exponent of each such term is less than or equal to $f_{\tau'}(\es)
= \eps_{g(\tau')} <_{\epsop_\X} \eps_{g(\tau)}$. So, the whole sum is
less than $\eps_{g(\tau)} = f_\tau(\si')$. Suppose now that
$\si'\neq\es$. Since the proof of Lemma \ref{lemma:monotone} (based on
the proof of Lemma \ref{h:monotone}) uses the monotonicity of
$f_{\tau'}$ only for strings shorter then $\si$ (by (\ref{P5})), we get
that $h_{f_{\tau'}}(\si') >_{\epsop_\X} h_{f_{\tau'}}(\si)$.
\end{proof}

\begin{theorem}\label{omega}
For every $Z \in \Bai$, there is a $Z$-computable linear ordering $\X$
such that the jump of every descending sequence in $\X$ computes
$Z^{(\om)}$, but there is a $Z$-computable descending sequence in
$\epsop_{\X}$.
\end{theorem}
\begin{proof}
Let $\X = \la \JT^\om (T_Z), {\KB} \ra$ where again $T_Z = \set{Z\upto
n}{n\in \N}$. By Lemma \ref{lemma:JTom comp}, $\X$ is $Z$-computable.
By Lemma \ref{lemma:JTom paths}, $\J^\om(Z)$ is the unique path in
$\JT^\om (T_Z)$. Therefore, by Lemma \ref{lemma:KB}, the jump of every
descending sequence in $\X$ computes $\J^\om(Z) \teq Z^{(\om)}$.

Let $g$ be the identity on $\X$, which is obviously
$(\supset,<_{\X})$-monotone. By Lemma \ref{lemma:om monotone},
$h_g^\om$ is $(\supset,<_{\epsop_{\X}})$-monotone. Since $h_g^\om$ is
computable, $\set{h_g^\om(Z\upto n)}{n\in\N}$ is a $Z$-computable
descending sequence in $\epsop_{\X}$.
\end{proof}

%%%%%%%%%%%%%
\subsection{Reverse mathematics results}\label{ssect:rmepsilon}

In this section, we work in the weak system \RCA. Therefore, we do not
have an operation that given $Z\in\Seq$, returns $\J(Z)$, let alone
$\J^{\om}(Z)$. However, the predicates with two variables $Z$ and $Y$
that say $Y=\J(Z)$ and $Y=\J^\om(Z)$ are arithmetic as witnessed by
Lemmas \ref {lemma:J approx J} and \ref {lemma:Jom approx Jom}. Notice
that if condition (2) of Lemma \ref {lemma:J approx J} holds, then
\RCA\ can recover the sequence of $t_i$'s in the definition of $\J(Z)$
and show that $\J(Z)$ is as defined in \ref {def:J}. Furthermore, \RCA\
can show that $\J(Z)\teq Z'$ and hence that \ACA\ is equivalent to
\RCA+$\forall Z\exists Y(Y=\J(Z))$, and \ACApr\ is equivalent to
\RCA+$\forall Z\forall n\exists Y(Y=\J^n(Z))$.

Also, if condition (2) of Lemma \ref {lemma:Jom approx Jom} holds, then
as in the proof of that lemma, in \RCA\ we can uniformly build
$\J^m(Z)$ as $\bigcup_{n}J^m(\si_n)$, and show that $\J^\om(Z)$ is as
defined in Definition \ref {def:Jom si}. Furthermore, we can prove
Lemma \ref {lemma: Z om jump} in \RCA: if $Y=\J^\om(Z)$, then $Y$ can
compute $Z^{(\om)}$, and if $X=Z^{(\om)}$, then $X$ can compute a real
$Y$ such that $Y=\J^\om(Z)$. Therefore, we get that \ACApl\ is
equivalent to \RCA+$\forall Z\exists Y(Y=\J^{\om}(Z))$.

We already know, from Girard's result Theorem \ref {Girard} that over
\RCA, the statement \lq\lq if \X\ is a well-ordering then $\omop^\X$ is
a well-ordering\rq\rq\ is equivalent to \ACA. We now start climbing up
the ladder.

\begin{theorem}\label{ACApr}
Over \RCA, $\forall n\, \WO(\X\mapsto\iexpop n\X)$ is equivalent to
\ACApr.
\end{theorem}
\begin{proof}
We showed, in Corollary \ref {cor forward ACApr},  that
\ACApr$\vdash\forall n\, \WO(\X\mapsto\iexpop n\X)$.

Suppose now that $\forall n\, \WO(\X\mapsto\iexpop n\X)$ holds.
Consider $Z\in\Bai$ and $n\in\om$; we want to show that $\J^n(Z)$
exists. By Girard's theorem we can assume \ACA. Let $\X_Z^n = \la T_n,
{\KB} \ra$, where $T_n = \JT^n(T_Z)$ as in the proof of Theorem
\ref{0^n}. The proof that there is a $Z$-computable descending sequence
in $\iexpop n{\X_Z^n}$ is finitary and goes through in \RCA. So, by
$\WO(\X\mapsto\iexpop n\X)$ we get a descending sequence in $\X_Z^n$.
By Lemma \ref {lemma:KB}, using \ACA, we get $Y_n \in [T_n]$. For each
$i\leq n$, let $Y_i=\K^{n-i}(Y_n)$. Lemma \ref {lemma:JT paths} shows
that for each $i$, $Y_i\in [T_i]$ and $Y_i=\J(Y_{i-1})$. Since $Z$ is
the only path through $T_Z$, we get that $Y_0=Z$, and so $Y_n=\J^n(Z)$.
\end{proof}

\begin{theorem}\label{rev:ACApl}
Over \RCA, $\WO(\X\mapsto\epsop_\X)$ is equivalent to \ACApl.
\end{theorem}
\begin{proof}
We already showed that \ACApl\ proves $\WO(\X\mapsto\epsop_\X)$ in
Corollary \ref{cor forward ACApl}.

Assume \RCA+$\WO(\X\mapsto\epsop_\X)$. Let $Z\in \Bai$; we want to show
that there exists $Y$ with $Y=\J^{\om}(Z)$. Build $\X = \la \JT^\om
(T_Z), {\KB} \ra$ as in Theorem \ref {omega}. The proof that
$\epsop_{\X}$ has a $Z$-computable descending sequence is completely
finitary and can be carried out in \RCA. By $\WO(\X\mapsto\epsop_\X)$,
we get that $\X$ has a descending sequence. Since we have \ACA\ we can
use this descending sequence to get a path $Y$ through $\JT^\om (T_Z)$.
Now, the proof of Lemma \ref {lemma:JTom paths} translates into a proof
in \RCA\ that $Y$ is $\J^\om$ of some path through $T_Z$. Since $Z$ is
the only path through $T_Z$, we get $Y=\J^\om(Z)$ as wanted.
\end{proof}

%%%%%%%%%%%%%%%%%%%%%%%%%%%%%%%%%%%%%%%%%%%%%%%%%%%%%%%%%
\section{General Case}\label{sect:General case}

In this section we define the $\om^\a$-Jump operator, the $\om^\a$-Jump
function, and the $\om^\a$-Jump Tree, for all computable ordinals $\a$.
The constructions of Sections \ref{sect:exp} and \ref{sect:eps}, where
we considered $\a=0$ and $\a=1$ respectively, are thus the simplest
cases of what we will be doing here.

The whole construction is by transfinite recursion, and the base case
was covered in Section \ref{sect:exp}. If $\a>0$ is a computable
ordinal, we assume that we have a fixed non-decreasing computable
sequence of ordinals $\set{\a_i}{i\in \N}$ such that
$\a=\sup_{i\in\N}(\a_i+1)$. (So, if $\a=\g+1$, we can take $\a_i=\g$
for all $i$.) Notice that we have $\sum_{i\in\N}\om^{\a_i}=\om^\a$. In
defining the $\om^\a$-Jump operator, the $\om^\a$-Jump function, and
the $\om^\a$-Jump Tree we make use the $\om^{\a_i}$-Jump operator, the
$\om^{\a_i}$-Jump function, and the $\om^{\a_i}$-Jump Tree for each
$i$.

\subsection{The iteration of the jump}\label{ssect:iterating jump}

Our presentation here is different from the one of previous sections,
where we defined the operator first. Here we start from the
$\om^\a$-Jump function, prove its basic properties, then use it to
define the $\om^\a$-Jump Tree, and eventually introduce the
$\om^\a$-Jump operator.

Let $\a>0$ be a computable ordinal and $\set{\a_i}{i\in \N}$ be its
canonical sequence as described above. To simplify the notation in the
definition of the $\om^\a$-Jump function, assume we already defined
$\Jom{\a_i}$ and $\Kom{\a_i}$ for all $i$, and let $\Jom\a_n\colon \Seq
\to \Seq$ and $\Kom\a_n\colon \Seq \to \Seq$ be defined recursively by
\begin{alignat*}{2}
\Jom\a_0 & = id; \qquad & \Jom\a_{n+1} & = \Jom{\a_n} \circ \Jom\a_n;\\
\Kom\a_0 & = id; \qquad & \Kom\a_{n+1} & = \Kom\a_n \circ \Kom{\a_n}.
\end{alignat*}
In other words:
\begin{align*}
\Jom\a_n &= J^{\om^{\a_{n-1}}}\circ J^{\om^{\a_{n-2}}} \circ \cdots \circ \Jom{\a_0}, \\
\Kom\a_n &= K^{\om^{\a_0}}\circ K^{\om^{\a_{1}}} \circ \cdots \circ \Kom{\a_{n-1}}.
\end{align*}

\begin{definition}
The \emph{$\om^\a$-Jump function} is the map $\Jom\a \colon \Seq \to
\Seq$ defined by
\[
\Jom\a(\si) = \la \Jom\a_1(\si)(0), \Jom\a_2(\si)(0), \dots,\Jom\a_{n-1}(\si)(0)\ra,
\]
where $n$ is least such that $\Jom\a_n(\si)=\es$. In this case, since
$\Jom\a_n(\si)=\Jom{\a_{n-1}}(\Jom\a_{n-1}(\si))$, by (\ref{Pa1}) below
applied to $\a_{n-1}$, we have $|\Jom\a_{n-1}(\si)| =1$.

Given $\tau\in \Jom\a(\Seq)$, let
\[
\Kom\a(\tau) = \Kom\a_{|\tau|}(\ell(\tau)).
\]
In particular $\Kom\a(\es)=\es$, since $\Kom\a_0$ is the identity
function.
\end{definition}

Since for $\a=1$ we have $\a_i=0$ for every $i$, the definitions we
just gave match exactly Definitions \ref{def:Jom si} and \ref{def:Kom
si}, where we introduced $J^\om$ and $K^\om$. We will not mention again
this explicitly, but the reader should keep in mind that the case
$\a=1$ of Section \ref{sect:eps} is the blueprint for the work of this
section.

Notice that, by transfinite induction, $\Jom\a$ and $\Kom\a$ are
computable.

The following properties generalize those of Lemmas \ref{lemma:J K} and
\ref{lemma:Jom Kom}. We will refer to them, as usual,  as (\ref{Pa1}),
\dots, (\ref{Pa7}).

\begin{lemma}  \label{lemma:J om al K}
For $\si, \tau'\in \Seq$, $\tau\in \Jom\a(\Seq)$,
\begin{enumerate}\renewcommand{\theenumi}{P$^{\om^\a}\!$\arabic{enumi}}
\item $\Jom\a(\si)=\es$ if and only if $|\si|\leq 1$. \label{Pa1}
\item $\Kom\a(\Jom\a(\si))= \si$ for $|\si|\geq 2$. \label{Pa2}
\item $\Jom\a(\Kom\a(\tau))=\tau$.        \label{Pa3}
\item If $\si\neq\si'$ and at least one has length $\geq 2$, then
    $\Jom\a(\si)\neq \Jom\a(\si')$.   \label{Pa4}
\item $|\Jom\a(\si)|<|\si|$ and $|\Kom\a(\tau)|>|\tau|$ except when
    $\tau=\es$.    \label{Pa5}
\item If $\tau'\subset\tau$ then $\tau'\in \Jom\a(\Seq)$ and
    $\Kom\a(\tau')\subset \Kom\a(\tau)$.      \label{Pa6}
\item If $\Jom\a(\si')\subseteq \Jom\a(\si)$ and $\a>0$ then for
    every $m$, $\Jom\a_m(\si')\subseteq \Jom\a_m(\si)$. \label{Pa7}
\end{enumerate}
\end{lemma}
\begin{proof}
The proof is by transfinite induction on $\a$. The case $\a=0$ is Lemma
\ref{lemma:J K}.

Since $\Jom\a(\si)=\es$ if and only if $\Jom{\a_0}(\si)=\es$,
(\ref{Pa1}) follows from the same property for $\a_0$.

To prove (\ref{Pa2}) let $|\Jom\a(\si)|=n-1>0$. Then $\ell(\Jom\a(\si))
= \la \Jom\a_{n-1}(\si)(0) \ra = \Jom\a_{n-1}(\si)$ because
$|\Jom\a_{n-1}(\si)|=1$ as noticed above. Since $\Kom\a(\Jom\a(\si)) =
\Kom\a_{n-1} (\Jom\a_{n-1}(\si))$, $\Kom\a(\Jom\a(\si)) = \si$ follows
from (\ref{Pa2}) for $\a_{n-2}, \a_{n-3}, \dots, \a_0$.

As in the proof of the case $\a=1$ in Lemma \ref{lemma:Jom Kom},
(\ref{Pa3}), (\ref{Pa4}) and (\ref{Pa5}) follow from the properties we
already proved.

The proof of (\ref{Pa6}) is also basically the same as the proof of
(\ref{Pom6}). We recommend the reader to have Figure \ref{fig tau =
Jomalpha si} in mind while reading the proof. The nontrivial case is
when $\tau' \neq \es$. Let $\si$ be such that $\tau=\Jom\a(\si)$. The
idea is to define $\si'\subset \si$ as in the picture and then show
that $\tau'=\Jom\a(\si')$. Notice that $|\si|>|\tau|\geq2$ by
(\ref{Pa5}), and that $\si=\Kom\a(\tau)$ by (\ref{Pa2}). Notice also
that $\ell(\tau') = \la \tau(|\tau'|-1)\ra = \la
\Jom\a_{|\tau'|}(\si)(0) \ra \subset \Jom\a_{|\tau'|}(\si)$, where the
strict inclusion is  because $|\tau'|<|\tau|$ and hence
$|\Jom\a_{|\tau'|}(\si)|>1$. By induction on $i \leq |\tau'|$ we can
show, using (\ref{Pa6}) and (\ref{Pa2}) for $\a_{|\tau'|-1}, \dots,
\a_1, \a_0$, that
\begin{align*}
(\Kom{\a_{|\tau'|-i}} \circ \dots \circ \Kom{\a_{|\tau'|-1}})(\ell(\tau'))
& \subset (\Kom{\a_{|\tau'|-i}} \circ \dots \circ \Kom{\a_{|\tau'|-1}})(\Jom\a_{|\tau'|}(\si))\\
& = \Jom\a_{|\tau'|-i}(\si)
\end{align*}
and $(\Kom{\a_{|\tau'|-i}} \circ \dots \circ \Kom{\a_{|\tau'|-1}})
(\ell(\tau')) \in \Jom\a_{|\tau'|-i}(\Seq)$. In particular, when
$i=|\tau'|$, if we set $\si' =\Kom\a_{|\tau'|}(\ell(\tau'))$, we obtain
$\si' \subset \si$. Furthermore, by (\ref{Pa2}) applied to
$\a_{0},\dots,\a_{|\tau'|-i-1}$, we also get
\begin{equation}\label{bb}
\Jom\a_{|\tau'|-i}(\si') = (\Kom{\a_{|\tau'|-i}} \circ \dots \circ \Kom{\a_{|\tau'|-1}})(\ell(\tau')) \subset \Jom\a_{|\tau'|-i}(\si).
\end{equation}
Therefore, for every $j<|\tau'|$
\[
\Jom\a(\si')(j)=\Jom\a_{j+1}(\si')(0)=\Jom\a_{j+1}(\si)(0)=\tau(j)=\tau'(j).
\]
Since $\Jom\a_{|\tau'|-1}(\si') = \ell(\tau')$ which has length 1, we
get that $\Jom\a(\si')$ has length $|\tau'|$ as wanted.

For (\ref{Pa7}) let $\tau'=\Jom\a(\si')$. Then, if $i=|\tau'|-m$,
equation \ref{bb} shows that $\Jom\a_m(\si')\subseteq \Jom\a_m(\si)$.
\end{proof}

\begin{figure}
{\Small
\[
\xymatrix@=10pt{
\tau \ar@{}|{=}[r] \ar@{}|(.4){\shortparallel}[dd]& \JJom\a(\si)  \ar@{}|(.4){\shortparallel}[dd] \\
\ar@{}|{\frown}[d]	&	\ar@{}|{\frown}[d] &\ar@{-}[dddddd]	&&   \emptyset \\
\tau(|\tau|-1)\ar@{}|{=}[r]	&\Jom\a_{|\tau|}(\si)(0)&	&\la \tau(|\tau|-1)\ra\ar@{}|(.6){=}[r]&   \Jom\a_{|\tau|}(\si)  \ar_ {\Jom{\a_{|\tau|}}}[u] \\
\vdots	& \vdots			&&&&	\ddots \ar_ {\Jom{\a_{|\tau|-1}}}[ul] \\
\tau(|\tau'|-1)\ar@{}|{=}[r]	& \Jom\a_{|\tau'|}(\si)(0)	&& \la \tau(|\tau'|-1)\ra\ar@{}|(.6){=}[r]  &\Jom\a_{|\tau'|}(\si') \ar@{}|(.8){\subset}[r]  \ar_ {\Kom{\a_{|\tau'|-1}}}[dr]  & \cdots  \ar@{}|(.3){\subset}[r] &  \Jom\a_{|\tau'|}(\si)  \ar_ {\Jom{\a_{|\tau'|}}}[ul] \\
\vdots& \vdots	&&&&\ddots \ar_ {\Kom{\a_1}}[dr]&		&\ddots  \ar_{\Jom{\a_{|\tau'|-1}}}[ul]\\
\tau(0)\ar@{}|{=}[r]& \Jom\a_1(\si)(0)	&&&&& \Jom\a_1(\si') \ar_ {\Kom{\a_0}}[dr]  \ar@{}|(.6){\subset}[r]    & \cdots  \ar@{}|(.3){\subset}[r]    & \Jom\a_1(\si) \ar_{\Jom{\a_1}}[ul]\\
\ar@{}|{\smile}[u]	&	\ar@{}|{\smile}[u]&&&&&		 &\si'	\ar@{}|(.6){\subset}[r]    & \cdots  \ar@{}|(.3){\subset}[r]       &  \si \ar_{\Jom{\a_0}}[ul]\\
}
\]}
\caption{Assuming $\tau=\Jom\a(\si)$ and $\tau'\subset\tau$.}\label{fig tau = Jomalpha si}
\end{figure}

We can now introduce the $\om^\a$-Jump Tree and prove its
computability.

\begin{definition}
Given a tree $T \subseteq \Seq$ the \emph{$\om^\a$-Jump Tree of $T$} is
\[
\JTom{\a}(T) = \set{\Jom\a(\si)}{\si\in T}.
\]
\end{definition}

\begin{lemma}\label{lemma:JTomegacomp al}
For every tree $T$, $\JTom{\a}(T)$ is a tree computable in $T$.
\end{lemma}
\begin{proof}
The proof is again the same as the one of Lemma \ref{lemma:JTcomp},
using Lemma \ref{lemma:J om al K} in place of Lemma \ref{lemma:J K}.
\end{proof}

We now define the $\om^\a$-Jump operator $\JJom\a\colon\Bai\to\Bai$ by
transfinite induction: the base case is the Jump operator $\J$
(Definition \ref{def:J}). Given $\a$ we assume that $\JJom{\a_n}$ has
been defined for all $n$. To simplify the notation let us define
$\JJom{\a}_n$ recursively by $\JJom\a_0  = id$, $\JJom\a_{n+1} =
\JJom{\a_n} \circ \JJom\a_n$, so that
\[
\JJom\a_n = \JJom{\a_{n-1}} \circ \JJom{\a_{n-2}} \circ \cdots \circ \JJom{\a_0}.
\]
%Similarly, set
%\[
%\KKom\a_n = \KKom{\a_0} \circ \cdots \circ \KKom{\a_{n-2}} \circ \KKom{\a_{n-1}}.
%\]

\begin{definition}\label{def: Ja Z}
Given the computable ordinal $\a$ we define the \emph{$\om^\a$-Jump
operator} $\JJom\a\colon \Bai\to\Bai$ and its inverse $\KKom\a$ by
\[
\JJom\a(Z)(n) = \JJom{\a}_{n+1}(Z)(0) \quad \text{ and } \quad
\KKom\a(Y) =  \bigcup_n \Kom\a(Y\upto n).
\]
\end{definition}

We first show that $\KKom\a$ is indeed the inverse of $\JJom\a$.

\begin{lemma}\label{lemma: Y equal Koma Z}
If $Y=\JJom\a(Z)$ then $Z=\KKom\a(Y)$.
\end{lemma}
\begin{proof}
The proof of the lemma is by transfinite induction. Let
$\set{\a_i}{i\in \N}$ be the fixed canonical sequence fo $\a$. Recall
from the definition of $\Kom\a$ that $\Kom\a(Y\upto n)  = \Kom\a_n(\la
Y(n-1)\ra)$. Since $\la Y(n-1)\ra=\la \JJom{\a}_{n}(Z)(0)\ra \subseteq
\JJom{\a}_{n}(Z)$, by the induction hypothesis applied to
$\a_{n-1},\dots,\a_0$, we get that $\Kom\a_n(\la Y(n-1)\ra) \subseteq
Z$. So $Z \supseteq \KKom\a(Y)$. By (\ref{Pa5}) applied to
$\a_0,\dots,\a_{n-1}$ we get that $|\Kom\a_n(\la Y(n-1)\ra)|> n+1$ and
hence $Z= \KKom\a(Y)$.
\end{proof}

\begin{lemma}\label{lemma: Ja Z vs Z to the om alpha}
For every $Z\in \Bai$ and computable ordinal $\a$, $\JJom\a(Z)\teq Z^{(\om^\a)}$.
\end{lemma}
\begin{proof}
This is again proved by transfinite induction. Assuming that
$\JJom{\a_i}(Z)\teq Z^{(\om^{\a_i})}$ for every $i$, and uniformly in
$i$, we immediately obtain $\JJom\a_n(Z) \teq Z^{(\b_n)}$, where $\b_n
= \sum_{i=0}^{n-1} \om^{\a_i}$, for every $n$. Since $\beta_n<\om^\a$,
$\JJom\a(Z)\leqt Z^{(\om^\a)}$ is immediate.

For the other reduction we need to uniformly compute $\JJom\a_n(Z)$
from  $\JJom\a(Z)$. The same way we compute $Z$ from $\JJom\a(Z)$
applying $\KKom\a$, we can compute $\JJom\a_n(Z)$ by forgetting about
$\a_0,\dots,\a_{n-1}$. In other words, by the same proof as Lemma \ref
{lemma: Y equal Koma Z} we can show that for every $m$
\[
\JJom\a_m(Z)= \bigcup_{n> m} \Kom{\a_{m}}(\Kom{\a_{m+1}}(\dots(\Kom{\a_{n-1}}(\la Y(n-1)\ra))\dots))
\]
using $\Kom{\a_{m}} \circ \Kom{\a_{m+1}} \circ\cdots\circ \Kom{\a_{n-1}}$ instead of $\Kom\a_n$.
\end{proof}

We can now prove that $\Jom\a$ approximates $\JJom\a$, extending Lemma
\ref{lemma:Jom approx Jom}.

%\begin{lemma}\label{lemma:Jomal approx Jomal}
%For every $Z \in \Bai$ and $n \in \N$ there exists $\si \subset Z$ with
%$|\si|>n$ such that $\JJom\a(Z) \upto n = \Jom\a(\si)$. Actually,
%$\si=\Kom\a(\JJom\a(Z) \upto n)$ when $n>0$.
%\end{lemma}
%\begin{proof}
%When $n=0$ let $\si=Z\upto1$, which works by (\ref{Pa1}).
%
%When $n>0$ we use transfinite induction on $\a$. If $\a=0$ the
%statement follows from one direction of Lemma \ref{lemma:J approx J}
%and its proof.
%
%Given $\a>0$, $Z$ and $n$, define recursively $\si_i \in \Seq$ for
%$i\leq n$ by setting $\si_0= \JJom\a_n(Z)\upto1$ and $\si_i=
%\Kom{\a_{n-i}}(\si_{i-1})$ when $i>0$. Notice that
%$\si_n=\Kom\a_n(\JJom\a_n(Z)\upto1)= \Kom\a(\JJom\a(Z) \upto n)$.
%Furthermore, by (\ref{Pa5}) for the $\a_{n-i}$s, $|\si_i|>i$. Applying
%the induction hypothesis to $\a_{n-1}, \dots, \a_0$ we obtain that for
%each $i\leq n$, $\si_i \subset \JJom\a_{n-i}(Z)$ and $(\Jom{\a_{n-i-1}}
%\circ \dots \circ \Jom{\a_{n-i'}}) (\si_{i'})=\si_i$ whenever $i<i'$.
%Setting $i'=n$, we obtain $\Jom\a_{n-i}(\si_n)= \si_i$ and hence
%$\Jom\a_{n-i}(\si_n)(0)= \si_i(0)= \JJom\a_{n-i}(Z)(0)$ for every
%$i<n$. Since $|\Jom\a_n(\si_n)(0)|=|\si_0|=1$ we have
%$|\Jom\a_{n-i}(\si_n)|=n$ and therefore $\Jom\a(\si_n) = \JJom\a(Z)
%\upto n$.
%\end{proof}

\begin{lemma}\label{lemma:Jomal approx Jomal}
Given $Y,Z\in \Bai$, the following are equivalent:
\begin{enumerate}
\item $Y = \JJom\a(Z)$;
\item for every $n$ there exists $\si_n\subset Z$ with $|\si_n|>n$ such
    that $Y\upto n =\Jom\a(\si_n)$.
\end{enumerate}\end{lemma}
\begin{proof}
We first prove (1) $\implies$ (2). When $n=0$ let $\si_0=Z\upto1$, which
works by (\ref{Pa1}).
Let $\si_n=\Kom\a(Y \upto n)$.
Then $\si_n \subseteq \KKom\a(Y)= Z$, and $Y\upto n =\Jom\a(\si_n)$.
We get that $|\si_n|>n$ by applying (\ref{Pa5}) $n$ times to $\si_n=\Kom\a_n(\la Y(n-1)\ra)$.

The proof of (2) $\implies$ (1) is similar to the proof of Lemma
\ref{lemma:Jom approx Jom} but uses transfinite induction. By
(\ref{Pa7}), for all $m$ and $n$, $\Jom\a_m(\si_n)\subseteq
\Jom\a_m(\si_{n+1})$, and hence we can consider $\bigcup_n
\Jom\a_m(\si_n) \in \Bai$. Then, by the induction hypothesis,
$\bigcup_n \Jom\a_m(\si_n) = \JJom\a_m(Z)$, and hence
$\Jom\a_m(\si_n)(0) =\JJom\a_m(Z)(0)$ for all $m<n$. It follows that
for every $m$ and $n>m$
\[
\JJom\a(Z)(m)=\JJom\a_{m+1}(Z)(0)=\Jom\a_{m+1}(\si_{n})(0) = \Jom\a(\si_{n})(m)= Y(m).\qedhere
\]
\end{proof}

We are now able to show the intended connection between the
$\om^\a$-Jump Tree and the $\om^\a$-Jump operator.

\begin{lemma}\label{lemma:JTomal paths}
For every tree $T$, $[\JTom\a(T)] = \set{\JJom\a(Z)}{Z \in [T]}$.
\end{lemma}
\begin{proof}
To prove $\set{\JJom\a(Z)}{Z \in [T]} \subseteq [\JTom\a(T)]$ we can
argue as in the proof of Lemma \ref{lemma:JT paths}, using Lemma
\ref{lemma:Jomal approx Jomal} in place of Lemma \ref{lemma:J approx
J}.

To prove the other inclusion, fix $Y \in [\JTom\a(T)]$. Arguing as in
the proof of Lemma \ref{lemma:JTom paths}, we first let
$\si_n=\Kom\a(Y\upto n)\in T$. Let $Z=\KKom\a(Y)=\bigcup_{n\in \N}\si_n
\in [T]$. We get that $Y=\JJom\a(Z)$ from Lemma \ref{lemma:Jomal approx
Jomal}.
\end{proof}

%%%%%%%%%%%%%%%%%%%
\subsection{Jumps versus Veblen}\label{ssect:JvsV}
First, we need to iterate the Jump Tree operator along a finite string.

\begin{definition}
If $T$ is a tree and $\tau \in \JTom{\a}(T)$  we define
\[
\JTom{\a}_\tau(T) = \set{\Jom\a_{|\tau|+1}(\si)}
{\si\in T \land \tau\subseteq\Jom\a(\si)}.
\]
\end{definition}

\begin{lemma}\label{lemma: JT tau gen recursive}
For $\tau\in \JTom{\a}(T)$,
\begin{align*}
\JTom{\a}_\es(T)   &=    \JTom{\a_0}(T)  \\
\JTom{\a}_{\tau\conc\la c\ra}  (T)  &=   \JTom{\a_{|\tau|+1}}(\JTom{\a}_\tau(T)_{\la c\ra}).
\end{align*}
($T_{\la c\ra}$ was defined in \ref{def:Tsigma}.)
\end{lemma}
\begin{proof}
Straightforward induction on $|\tau|$.
\end{proof}

\begin{lemma}
Given a tree $T \subseteq \Seq$, $\tau\in \Seq$, and $c\in\N$
\[
\tau\conc \la c\ra  \in \JTom{\a}(T)       \iff      \la c \ra \in \JTom{\a}_{\tau}(T).
\]
\end{lemma}
\begin{proof}
Follows from the definitions of $\JTom{\a}(T)$ and $\JTom{\a}_{\tau}(T)$.
\end{proof}

We now generalize the construction of Definition \ref{def:fgtau}, by
defining an operator that converts a function with domain
$\JTom{\a}(T)$ and values in $\X$ into a function with domain $T$ and
values in $\phiop(\a,\X)$. We will show in Lemma \ref{lemma:halphag
monotone} that this operator preserves monotonicity.

\begin{definition}\label{def:halphag}
By transfinite recursion, we build, for each computable ordinal $\a$,
an operator $\hom{\a}$ such that given a linear ordering $\X$ and a
function
\[
g\colon \JTom{\a}(T) \to \X,
\]
it returns
\[
\hom{\a}_g\colon T\to \phiop(\a,\X).
\]

For $\a=0$, we let $\hom{\a} = h$ of Definition \ref{def:hgJ}. For
$\a>0$ we first define simultaneously for each $\tau \in \JTom{\a}(T)$
a function
\[
f_\tau \colon \JTom{\a}_\tau (T) \to \phiop(\a, \X)
\]
by recursion on $|\si|$:
\[
f_\tau (\si) =
\begin{cases}
\varphi_{\a, g(\tau)} & \text{if $\si = \es$;}\\
\\
\hom{\a_{n}}_{f_{\tau'}} (\si) & \text{if $\si \neq \es$, where $\tau'= \tau \conc \la \si(0) \ra$ and $n=|\tau'|$.}
\end{cases}
\]

We then define
\[
\hom{\a}_g = \hom{\a_0}_{f_\es}\colon T \to \phiop(\a,\X).
\]
\end{definition}

\begin{lemma}\label{lemma:halphag monotone}
If $g\colon \JTom{\a} (T) \to \X$ is total and $(\supset,
<_\X)$-monotone, then $\hom{\a}_g\colon T\to \phiop(\a,\X)$ is also
total and $(\supset, <_{\phiop(\a,\X)})$-monotone. Moreover,
$\hom{\a}_g$ is computable in $g$.
\end{lemma}
\begin{proof}
We say that a partial function $e$ on a tree $T$ is {\em
$(n,T,\X)$-good} if $e$ is defined on all strings of length less than
or equal to $n$, it takes values in $\X$, and is
$(\supset,<_\X)$-monotone on strings of length less than or equal to
$n$.

By transfinite induction on $\a$ we will show that for every $n\in\N$
and every $(n,\JTom{\a}(T),\X)$-good  partial function $g$, we have
that $\hom{\a}_g$ is $(n+1,T,\phiop(\a,\X))$-good.
%This implies also that $\hom{\a}_g$ is computable in $g$.

For $\a=0$, this follows from the proof of Lemma \ref{lemma:monotone}:
recall that $h_g(\si)$ is a finite sum of terms of the form
$\om^{g(J(\si)\upto i)}$, and $|J(\si)|<|\si|$ by (\ref{P5}). Thus to
compute and compare $h_g$ on strings of length $\leq n+1$, we need only
$g$ to be defined and $(\supset,<_\X)$-monotone on strings of length
$\leq n$.

Now fix $\a>0$ and suppose that $g$ is $(n,\JTom{\a}(T),\X)$-good.
Since $\hom{\a}_g = \hom{\a_0}_{f_\es}$, by the induction hypothesis it
is enough to show that $f_\es$ is $(n,\JTom{\a}_\es(T),
\phiop(\a,\X))$-good. Notice that if $f_\es$ takes values in
$\phiop(\a,\X)$, then $\hom{\a}_g$ takes values in
$\phiop(\a_0,\phiop(\a,\X)) = \phiop(\a,\X)$ (by Definition
\ref{def:phiop}). We will prove by induction on $m\leq n$ that for
every $\tau\in \JTom\a(T)$ of length $n-m$, $f_\tau$ is
$(m,\JTom{\a}_\tau(T), \phiop(\a,\X))$-good. When $m=0$, all we need to
observe is that $f_\tau(\es)=\varphi_{\a,g(\tau)}\in \phiop(\a,\X)$,
and $g(\tau)$ is defined because $|\tau|=n$. Consider now $\tau\in
\JTom\a(T)$ of length $n-(m+1)$. If $\si=\es$, then $f_\tau(\es)$ is
correctly defined as in the case $m=0$. For $\si \in \JTom\a_\tau(T)$
with $0<|\si| \leq m+1$, let $\tau'=\tau\conc\la\si(0)\ra$. We first
need to check that $f_\tau(\si)= \hom{\a_{n-m}}_{f_{\tau'}}(\si)$ is
defined. By the subsidiary induction hypothesis $f_{\tau'}$ is
$(m,\JTom\a_{\tau'}(T),\phiop(\a,\X))$-good. By Lemma \ref{lemma: JT
tau gen recursive}, $\JTom\a_{\tau'}(T) = \JTom{\a_{n-m}}
(\JTom\a_\tau(T)_{\la\si(0)\ra})$. By the transfinite induction
hypothesis (since $\a_{n-m}<\a$) $\hom{\a_{n-m}}_{f_{\tau'}}$ is
$(m+1,\JTom\a_\tau(T)_{\la\si(0)\ra},\phiop(\a,\X))$-good. Therefore
$f_\tau(\si)$ is defined.
%Then $f_\tau(\si)=\hom{\a_{n-m}}_{f_{\tau'}}
%(\si)$ which is defined because since $f_{\tau'}$ is
%$(m,\JTom{\a}_{\tau'}, \phiop(\a,\X))$-good, and $\JTom{\a}_{\tau'}
%(T)  = \JTom{\a_{n-m}}(\JTom{\a}_\tau(T)_{\la \si(0)\ra})$, and the
%induction hypothesis on $\a_{n-m}$, $\hom{\a_{n-m}}_{f_{\tau'}}$ is
%$(m+1,\JTom{\a}_{\tau}, \phiop(\a,\X))$-good. $f_\tau$ takes values in
%$\phiop(\a,\X)$ because $\phiop(\a_{n-m},\phiop(\a,\X))=\phiop(\a,\X)$.
Now we need to show $f_\tau$ is
$(\supset,<_{\phiop(\a,\X)})$-monotone on strings of length less than
or equal to $m+1$. Take $\si'\subset \si \in \JTom{\a}_{\tau}(T)$
with $|\si|\leq m+1$. Again let $\tau'=\tau\conc\la\si(0)\ra$. By the
transfinite induction hypothesis, we know that
$\hom{\a_{n-m}}_{f_{\tau'}}$ is $(m+1,\JTom{\a}_{\tau}(T),
\phiop(\a,\X))$-good. Furthermore $f_{\tau'}$ is $(\supset,
<_{\phiop(\a,\X)})$-monotone and takes values in $\phiop(\a,\X) \upto
(\varphi_{\a,g(\tau')}+1)$, because $f_{\tau'}(\es)=
\varphi_{\a,g(\tau')}$. Therefore $\hom{\a_{n-m}}_{f_{\tau'}}$ takes
values below $\varphi_{\a_{n-m}}(\varphi_{\a,g(\tau')}+1)$. When
$\si'=\es$, $f_\tau(\si')=\varphi_{\a,g(\tau)}>
\varphi_{\a_{n-m}}(\varphi_{\a,g(\tau')}+1)> f_\tau(\si)$. When
$\si'\neq\es$, we use the monotonicity of
$\hom{\a_{n-m}}_{f_{\tau'}}$.
\end{proof}

\begin{theorem}\label{alpha}
For every computable ordinal $\a$ and $Z \in \Bai$, there exists a
$Z$-computable linear ordering $\X$ such that the jump of every
descending sequence in $\X$ computes $Z^{(\om^\a)}$, but there is a
$Z$-computable descending sequence in $\phiop(\a,{\X})$.
\end{theorem}
\begin{proof}
Let $\X = \la \JTom{\a} (T_Z), {\KB} \ra$ where $T_Z$ is the tree
$\set{Z\upto n}{n\in \N}$. By Lemma \ref{lemma:JTomegacomp al}, $\X$ is
$Z$-computable. By Lemma \ref{lemma:JTomal paths}, $\JJom\a(Z)$ is the
unique path in $\JTom{\a} (T_Z)$. Therefore, by Lemma \ref{lemma:KB},
the jump of every descending sequence in $\X$ computes $\JJom\a(Z)$ and
hence, by Lemma \ref{lemma: Ja Z vs Z to the om alpha}, computes
$Z^{(\om^\a)}$.

Let $g$ be the identity on $\X$, which is $(\supset,<_{\X})$-monotone.
By Lemma \ref{lemma:halphag monotone}, $h^\a_g$ is $(\supset,
<_{\phiop(\a,{\X})})$-monotone and computable. Thus $\set{h^\a_g(Z\upto
n)}{n\in\N}$ is a $Z$-computable descending sequence in
$\phiop(\a,{\X})$.
\end{proof}

\subsection{Reverse mathematics results}\label{ssect:rmVeblen}

In this section, we work in the weak system \RCA. Therefore, again, we
do not have an operation that given $Z\in\Bai$, returns $\JJom\a(Z)$
but the predicate with three variables $Z$, $Y$ and $\a$ that says
$Y=\JJom\a(Z)$ is arithmetic as witnessed by Lemma \ref {lemma:Jomal
approx Jomal}. Notice that if if we have that condition (2) of Lemma
\ref {lemma:Jomal approx Jomal} holds, then \RCA\ can recover all the
$\JJom\a_m(Z)$ and show that $\JJom\a(Z)$ is as defined in Definition
\ref {def: Ja Z}. We can then prove Lemma \ref {lemma: Ja Z vs Z to the
om alpha} in \RCA: if $Y=\JJom\a(Z)$, then $Y$ can compute
$Z^{(\om^\a)}$, and $Z^{(\om^\a)}$ can compute a real $Y$ such that
$Y=\JJom\a(Z)$. Therefore, we get that $\Pi^0_{\om^\a}$-\CA\ is
equivalent to \RCA+$\forall Z\exists Y(Y=\JJom\a(Z))$ and that \ATR\ is
equivalent to \RCA+$\forall \a \forall Z\exists Y(Y=\JJom\a(Z))$.

\begin{theorem}
Let $\a$ be a computable ordinal. Over \RCA,
$\WO(\X\mapsto\phiop(\a,\X))$ is equivalent to $\Pi^0_{\om^\a}$-\CA.
\end{theorem}
\begin{proof}
We already showed that
$\Pi^0_{\om^\a}$-\CA$\vdash\WO(\X\mapsto\varphi(\a,\X))$ in Corollary
\ref {cor forward Pi0alpha}.

The proof of the other direction is just the formalization of Theorem \ref {alpha} exactly as we did in Theorem \ref {rev:ACApl}.
\end{proof}

We now give a new, purely computability-theoretic,  proof of Friedman's
theorem.

\begin{theorem}
Over \RCA, $\WO(\X\mapsto\phiop(\X,0))$ is equivalent to \ATR.
\end{theorem}
\begin{proof}
We already showed that \ATR\ proves $\WO(\X\mapsto\phiop(\X,0))$ in
Corollary \ref{cor forward ATR}.

For the reversal, we argue within \RCA. Let $\a$ be any ordinal. Notice
that relative to the presentation of $\a$, all the constructions of
this section can be done as if $\a$ were any computable ordinal.
Therefore, by the previous theorem it is enough to show that
$\WO(\X\mapsto\phiop(\a,\X))$  holds. Let $\X$ be a well-ordering. We
now claim that $\phiop(\a,\X)$ embeds in $\phiop(\a+\X,0)$, which would
imply that $\phiop(\a,\X)$ is well-ordered too as needed to show
$\WO(\X\mapsto\phiop(\a,\X))$.

Define $f\colon \phiop(\a,\X)\to\phiop(\a+\X,0)$ by induction on the
terms of $\phiop(\a,\X)$, setting
\begin{itemize}
\item $f(0)=0$,
\item $f(\varphi_{\a,x}) = \varphi_{\a+x}(0)$,
\item $f(t_1+t_2)=f(t_1)+f(t_2)$,
\item $f(\varphi_a(t))=\varphi_a(f(t))$.
\end{itemize}

The proof that $f$ is an embedding is by induction on terms. Consider
$t,s\in \phiop(\a,\X)$. We want to show that $t\leq_{\phiop(\a,\X)}
s\iff f(t)\leq_{\phiop(\a+\X,0)} f(s)$. By induction hypothesis, assume
this is true for pairs of terms shorter than $t+s$. Suppose that $t\leq
s$. Using the induction hypothesis, it is not hard to show that
$f(t)\leq f(s)$. Suppose now that $t\not\leq s$. Then, none of the
conditions of Definition \ref{def:phiop} hold. If $t=0, t=t_1+t_2$, or
$t= \varphi_a(t_1)$, then we can apply the induction hypothesis again
and get that none of the conditions of Definition \ref{def:phiop} hold
for $f(t)$ and $f(s)$ either and hence $f(t)\not\leq f(s)$. The case
$t=\varphi_{\a,x}$ is the only one that deserves attention. In this
case we have that for no $y\geq x$, $\varphi_{\a,y}$ appears in $s$. It
then follows that $f(s)\in \phiop(\a+\X\upto x,0)$. Since
$f(t)=\varphi_{\a+x}(0)$ is greater than all the elements of
$\phiop(\a+\X\upto x,0)$ we get $f(t)\not\leq f(s)$ a wanted.
\end{proof}

\bibliography{Veblen}
\bibliographystyle{alpha}

\end{document}